\documentclass[final]{siamltex}

\newtheorem{remark}{Remark}[section]

\usepackage{times,amsmath,amsbsy,amssymb}
\usepackage{graphicx}

\title{CONVERGENCE OF ADAPTIVE MIXED FINITE
ELEMENT METHOD FOR CONVECTION-DIFFUSION-REACTION
EQUATIONS\thanks{This work was supported in part by The Natural
Science Foundation of Chongqing city under Grant No. CSTC,
2010BB8270, The Education Science Foundation of Chongqing
(KJ120420), The National Natural Science Foundation of China
(11171239),  The Project-sponsored by Scientific Research Foundation for the Returned Overseas Chinese
Scholars and Open Fund of Key Laboratory of Mountain Hazards and Earth Surface Processes, CAS.}}
\author{Shaohong Du\thanks{School of Science, Chongqing Jiaotong University,
                           Chongqing 400047, China, ({\it
                           dushhong@gmail.com}).}
\and Xiaoping XIE\thanks{Corresponding author. School of Mathematics, Sichuan University,
Chengdu 610064, China ({\it xpxiec@gmail.com}). }}
\begin{document}

\maketitle
\begin{small}
  {\bf{Abstract.}  \rm{We prove the convergence of an adaptive mixed finite
element method (AMFEM) for (nonsymmetric)
convection-diffusion-reaction equations. The convergence result
holds from the cases where convection or reaction is not present to
convection-or reaction-dominated problems. A novel technique of
analysis is developed without any quasi orthogonality for stress and
displacement variables, and without marking the oscillation dependent on
discrete solutions and data. We show that AMFEM
is a contraction of the error of the stress and displacement
variables plus some quantity. Numerical experiments confirm the
theoretical results.}}
\end{small}

\begin{it} Key words.\end{it} adaptive mixed finite element method,
quasi orthogonality, oscillation, convergence

\begin{it} AMS subject classifications.\end{it} 65N30, 65N15,
65N12, 65N50

\pagestyle{myheadings} \thispagestyle{plain} \markboth{}{CONVERGENCE OF ADAPTIVE MIXED FINITE ELEMENT METHOD}

\section {Introduction and main results}
Let $\Omega$ be a bounded polygonal or polyhedral domain in
${\mathbb{R}}^{d},d=2\ {\rm or}\ 3$. We consider the following
convection-diffusion-reaction equations:
\begin{equation}\label{equation1}
 \left \{ \begin{array}{ll}
    -\nabla\cdot(S\nabla p)+\nabla\cdot(p{\bf w})+rp  =f \quad   \mbox{in}\ \ \Omega,\\
    \hspace{41mm} p=0  \quad \mbox{on}\ \partial{\Omega},
 \end{array}\right.
\end{equation}
where $S\in L^{\infty}(\Omega;{\mathbb{R}}^{d\times d})$ is an
inhomogeneous and anisotropic diffusion-dispersion tensor, ${\bf w}$
is a (dominating) velocity field, $r$ is a reaction function, and
$f$ is a source term. The choice of homogeneous boundary conditions
is made for ease of presentation, since similar results are valid
for other boundary conditions.

Adaptive methods for the numerical solution of PDEs
 are now standard tools in science and
engineering to achieve better accuracy with minimum degrees of
freedom. The adaptive procedure of (\ref{equation1}) consists of
loops of the form
\begin{equation}\label{equation1+}
{\rm SOLVE}\rightarrow {\rm ESTIMATE}\rightarrow {\rm
MARK}\rightarrow {\rm REFINE}.
\end{equation}
A posteriori error estimation (ESTIMATE) is an essential ingredient
of adaptivity, and reaches its mature level after two decades of
development \cite{Anis;Oden,Babuska;Rheinboldt1,Babuska;Rheinboldt2,
Babuska;Strouboulis,
Carstensen;Hu,C-Hu-Orlando,Du;Xie,Verfurth}. However, the analysis of convergence of the whole
algorithm (\ref{equation1+}) is still in its infancy, and is carried
out mainly for standard adaptive finite element methods (AFEM) \cite{Cascon;Kreuzer;Nochetto;Siebert,Dorfler,Morin2000,Morin2003,Morin;Nochetto;Siebert}.

Due to the saddle-point characteristic of mixed finite element   approximation, there is no orthogonality
available, as is one of main difficulties in the convergence analysis of
AMFEM.  Thus one has to find some quasi-orthogonality instead of the
orthogonality,  and the occurrence of oscillation of data is inevitable.
Hence, how to deal with the oscillation becomes a key issue  in the analysis. For the convergence of AMFEM, the present studies mainly
focus   on Poisson equations.  In
\cite{Carstensen;Hoppe}, {\it Carstensen} and {\it Hoppe} proved the
error reduction and convergence for  the lowest-order
Raviart-Thomas element with marking the data oscillation.
{\it Chen}, {\it Holst} and {\it Xu} \cite{Chen;Holst;Xu} showed the
convergence of a quasi-error with marking the data oscillation.
In \cite{Becker,Carstensen;Rabus,Du2}, the convergence was analyzed
for   the lowest-order Raviart-Thomas element where the local
refinement was performed by  using only
either the estimators or
the data oscillation term.

 For general diffusion problems and more general
mixed elements, by using the orthogonality of the divergence of the flux, {\it Du} and {\it Xie} \cite{Du3}  showed the convergence of the flux error plus some quantity
   without marking the oscillation.

The purpose of this paper is to prove the following convergence
results for an AMFEM for the convection-diffusion-reaction equations
(\ref{equation1}) and verify them computationally.
\begin{theorem}\label{theorem--}{\rm(Convergence of AMFEM)}\ Denote by
$\{\mathcal{T}_{k},e_{k}^{2},A_{k}^{2}\}_{k\geq0}$ the sequence of
meshes, the error of the stress and displacement variables, and some
quantity (defined by (\ref{quantity})) produced by the AMFEM algorithm.
Let $h_{0}$ be  the mesh size of the quasi-uniform initial mesh
$\mathcal{T}_{0}$.  Then there exist two positive constants $q$ and $\alpha\in(0,1)$ such that
\begin{equation*}
e_{k+1}^{2}+(1-h_{0}q)A_{k+1}^{2}\leq\alpha^{2}(e_{k}^{2}+(1-h_{0}q)A_{k}^{2})
\end{equation*}
when
$h_{0}\leq\frac{1-\alpha^{2}}{1+\alpha^{2}}\frac{1}{q}$. This
means that AMFEM, as $h_0$ is small enough, converges with a linear rate $\alpha$, namely,
\begin{equation*}
e_{k}^{2}+(1-h_{0}q)A_{k}^{2}\leq\alpha^{2k}(e_{0}^{2}+(1-h_{0}q)A_{0}^{2}).
\end{equation*}
\end{theorem}

This theorem     extends the convergence results in \cite{Du3} in the following
several
aspects.
\begin{itemize}
\item We deal with more general
convection-diffusion-reaction equations here with variable coefficients
$S,{\bf w}$ and $r$, whereas in \cite{Du3} ${\bf w}$ and $r$ vanish.
\item The orthogonality for the divergence of the flux is
absent due to the convection term ${\bf w}\cdot\nabla p$ and the
zero order term $(r+\nabla\cdot{\bf w})p$.  So this contribution considers not
only
the flux (stress variable) error but also the displacement variable error.
\item The quasi-orthogonality for stress and displacement
variables also fails due to the  terms ${\bf w}\cdot\nabla p$
and  $(r+\nabla\cdot{\bf w})p$.  This will lead to an
additional constraint on the  mesh size, $h_0$, of the  quasi-uniform initial  mesh
$\mathcal{T}_{0}$.
\item The oscillation term depends on the discrete solution
and data. Therefore, the oscillation and error can not be
reduced separately here.  In \cite{Du3} the oscillation term is not included in the a
posteriori indicators.

\item Since the error and oscillation are now coupled, in
order to prove convergence without marking the oscillation term, we
need to handle them together by following the same  idea as
in \cite{Demlow, Hu2}.

\item In comparison with  previous analysis methods, we consider the a posteriori indicators with weighted
factors. We also release the constraint that the divergence of the
convection term is free in contrast to the analysis of standard AFEM (see
\cite{Mekchay;Nochetto}).
\end{itemize}

The rest of this paper is organized as follows. Section 2 gives some
preliminaries and details on notations. Section 3 derives an estimate
for the error between $L^{2}-$projection of the displacement and its
approximation solution, which is  key to the convergence analysis. Section 4 shows the
estimator reduction. We prove theorem
\ref{theorem--} (Convergence of AMFEM algorithm) in section 5 and present four numerical experiments to illustrate
properties of AFMEM in section 6.
\section{Assumptions, weak problem, and AMFEM algorithm}\ For a
domain $A\subset\mathbb{R}^{d}$, we denote by $L^{2}(A)$ and ${\bf
L}^{2}(A) =:(L^{2}(A))^{d}$ the spaces of square-integrable
functions, by $(\cdot,\cdot)_{A}$ the $L^{2}(A)$ or ${\bf L}^{2}(A)$
inner product, by $||\cdot||_{A}$ the associated norm, and by $|A|$
the Lebesgue measure of $A$. Let $H^{k}(A)$ be the usual
Sobolev space equipped with norm $||\cdot||_{k,A}$ for $k=1,2$;
$H_0^1(A):=\{v\in H^{1}(A):\ v|_{\partial A}=0\}$, $H({\rm
div},A):=\{{\bf v}\in {\bf L}^{2}(A):{\rm div}\ {\bf v}\in
L^{2}(A)\}$. $<\cdot,\cdot>_{\partial A}$ denotes $d-1$-dimensional
inner product on $\partial A$ for the duality paring between
$H^{-1/2}(\partial A)$ and $H^{1/2}(\partial A)$. In what follows we shall
omit the subscript $\Omega$ when $A=\Omega$.

Let $\mathcal{T}_{h}$ be a shape regular triangulation in the sense
of \cite{Ciarlet}, and denote the mesh size $h_{T}:=|T|^{1/d}$ with
$|T|$ the area of $T\in \mathcal{T}_{h}$. Let $C_{Q}$ be a positive constant depending
only on a quantity $Q$, and $C_{i} (i=1,2,\cdots)$ positive
constants determined only by the shape regularity of
$\mathcal{T}_{h}$. We denote by $\varepsilon_{h}$ the set of element
sides in $\mathcal{T}_{h}$, by $\varepsilon_{h}^{0}$ the set of
interior sides of elements. For $K\in\mathcal{T}_{h}$, denote by
$\varepsilon_{K}$ the set of sides of $K$. Furthermore, we denote by
$\omega_{K}$ and $\omega_{\sigma}$ the unions of all elements in $\mathcal{T}_{h}$ respectively sharing
a side with $K$ and sharing
 a side $\sigma\in\varepsilon_{h}$.
 We   use the
"broken Sobolev space" $H^{1}(\bigcup\mathcal{T}_{h}) :=\{\varphi\in
L^{2}(\Omega):\varphi|_{K}\in H^{1}(K),\forall
K\in\mathcal{T}_{h}\}$. $H^{2}(\bigcup\mathcal{T}_{h})$ is defined
analogously. Denote by $[v]|_{\sigma}
:=(v|_{K})|_{\sigma}-(v|_{L})|_{\sigma}$ the jump of $v\in
H^{1}(\bigcup\mathcal{T}_{h})$ over an interior side $\sigma :=K\cap
L$ of diameter $h_{\sigma}: ={\rm diam}(\sigma)$, shared by the two
neighboring (closed) elements $K,L\in\mathcal {T}_{h}$. Especially,
$[v]|_{\sigma} :=(v|_{K})|_{\sigma}$ if
$\sigma\subset\partial K\cap\partial\Omega$. Note that $[\cdot]$ is a linear
operator over the broken Sobolev space
$H^{1}(\bigcup\mathcal{T}_{h})$.

We note that throughout the paper,
the local version of differential operator $\nabla$ is understood in
the distribution sense, namely,
$\nabla_{h}:H^{1}(\bigcup\mathcal{T}_{h})\rightarrow
(L^{2}(\Omega))^{d}$ is defined with $\nabla_{h}v|_{K}
:=\nabla(v|_{K}) $ for all $K\in\mathcal{T}_{h}$.

Given a unit normal vector ${\bf n}_{\sigma}=(n_1,\cdots,n_d)^T$ along the side
$\sigma$ with $d=2,3$, we define the tangential component of a vector ${\bf
v}\in\mathbb{R}^{d}$ with respect to ${\bf n}_{\sigma}$ by
\begin{equation*}
\gamma_{{\bf t}_{\sigma}}({\bf v}) :=\left \{ \begin{array}{ll}
{\bf v}\cdot (-n_2,n_1)^T \quad   \mbox{if}\ \ d=2,\\
{\bf v}\times{\bf n}_{\sigma} \quad   \mbox{if}\ \ d=3,
 \end{array}\right.
\end{equation*}
where  
$\times$ denotes
the usual vector product of two vectors in $\mathbb{R}^{3}$.

Following \cite{Vohralik1}, we suppose that there exists an original
triangulation $\mathcal{T}_{0}$ of $\Omega$ such that  data of
the problem ($\ref{equation1}$) are given in the
following way.\\
\textbf{Assumptions of data }:\\
{\it (D1)\ $S_{K} :=S|_{K}$ is a constant, symmetric, and uniformly
positive definite tensor such that $c_{S,K}{\bf v}\cdot{\bf v}\leq
S_{K}{\bf v}\cdot{\bf v}\leq C_{S,K}{\bf v}\cdot{\bf
v}$ holds  for all ${\bf v}\in\mathbb{R}^{d}$ and all
$K\in\mathcal{T}_{0}$ with $c_{S,K}>0,C_{S,K}>0$;\\
(D2)\ ${\bf w}\in RT_{0}(\mathcal {T}_{0})$ (see below) and
$|{\bf w}|_{K}|\leq C_{{\bf w},K}$ for all
$K\in\mathcal{T}_{0}$ with $C_{{\bf w},K}\geq0$;\\
(D3)\ $r_{K} :=r|_{K}$ is a constant for all
$K\in\mathcal{T}_{0}$;\\
(D4)\ $c_{{\bf w},r,K} :=1/2\nabla\cdot{\bf w}|_{K}+r|_{K}\geq0$ and
$C_{{\bf w},r,K} :=|\nabla\cdot{\bf w}|_{K}+r_{K}|$ for all $K\in\mathcal{T}_{0}$;\\
(D5)\ $f\in L^{2}(\Omega)$;\\
(D6)\ if $c_{{\bf w},r,K}=0$, then $C_{{\bf w},r,K}=0$.}

Note that  in \cite{D1,D2}  $f|_{K}$ is assumed to be a polynomial of degree at most $k$ for each
$K\in\mathcal{T}_{0}$ so as to derive the efficiency of the residual indicators. Here we relax the restriction of $f$ (cf.  (D5)).

Introduce the
stress variable ${\bf u} :=-S\nabla p$, the mixed variatinal problem of (\ref{equation1}) reads
as: Find $({\bf u},p)\in H({\rm div},\Omega)\times L^{2}(\Omega) $
such that
\begin{equation}\label{equation-1}
(S^{-1}{\bf u},{\bf v})-(p,\nabla\cdot{\bf v})=0\ \ \ {\rm for\
all}\ {\bf v}\in H({\rm div},\Omega),
\end{equation}
\begin{equation}\label{equation0}
(\nabla\cdot{\bf u},\varphi)-(S^{-1}{\bf u}\cdot{\bf
w},\varphi)+((r+\nabla\cdot{\bf w})p,\varphi)=(f,\varphi)\ \ {\rm
for\ all}\ \ \varphi\in L^{2}(\Omega).
\end{equation}

Let
$P_{0}(K)$ denote the set of constant functions on each
$K\in\mathcal{T}_{h}$.  We respectively define the lowest order
Raviart-Thomas finite element (\cite{Raviart and Thomas})  space and the
piecewise constant space as following:
$$
RT_{0}(\mathcal {T}_{h}) :=\left\{ \begin{array}{c} {\bf q}_{h}\in
{\bf H}({\rm div},\Omega):\ \forall K\in\mathcal {T}_{h},\  \exists
{\bf a}\in\mathbb{R}^{d},\  \exists b\in\mathbb{R}, \\ {\rm such\
that}\ {\bf q}_{h}({\bf x})={\bf a}+b{\bf x},{\rm for\ all}\ {\bf
x}\in K.
\end{array}\right\}
$$
and
$$
P_{0}({\mathcal {T}_{h}}) :=\{v_{h}\in L^{\infty}(\Omega):\ \forall
K\in\mathcal {T}_{h},\  v_{h}|_{K}\in P_{0}(K) \}.
$$
We note that $\nabla\cdot
(RT_{0}(\mathcal{T}_{h}))\subset P_{0}({\mathcal {T}_{h}})$.

The centered mixed finite element scheme (cf.
\cite{Douglas,Vohralik1}) of (\ref{equation1}) reads as: Find $({\bf
u}_{h},p_{h})\in RT_{0}(\mathcal {T}_{h})\times P_{0}({\mathcal
{T}_{h}})$ such that
\begin{equation}\label{equation2}
(S^{-1}{\bf u}_{h},{\bf v}_{h})-(p_{h},\nabla\cdot{\bf v}_{h})=0\ \
\ {\rm for\ all}\ {\bf v}_{h}\in RT_{0}(\mathcal {T}_{h}),
\end{equation}
\begin{equation}\label{equation3}
(\nabla\cdot{\bf u}_{h},\varphi_{h})-(S^{-1}{\bf u}_{h}\cdot{\bf
w},\varphi_{h})+((r+\nabla\cdot{\bf
w})p_{h},\varphi_{h})=(f,\varphi_{h})\ \ {\rm for\ all}\ \
\varphi_{h}\in P_{0}({\mathcal {T}_{h}}).
\end{equation}

In what follows, we shall show an AMFEM algorithm based on the a posteriori error estimator  developed in \cite{D1}. We note that our convergence analysis below is also valid for AMFEM based on the estimator proposed in \cite{D2} .

 Suppose that the
module SOLVE outputs a pair of discrete solutions over
$\mathcal{T}_{h}$, namely, $({\bf u}_{h},p_{h})={\rm
SOLVE}(\mathcal{T}_{h})$.
The  estimator
  in \cite{D1}  consists of   several
indicators with different weight factors, where the elementwise
estimator $\eta_{\mathcal{T}_{h}}^{2}({\bf u}_{h},p_{h}, K)$ can, for convenience, be  abbreviated to
\begin{equation*}
\displaystyle\eta_{\mathcal{T}_{h}}^{2}({\bf u}_{h},p_{h}, K)
:=D_{K}^{2}h_{K}^{2}||S^{-1}{\bf
u}_{h}||_{K}^{2}+\alpha_{K}^{2}h_{K}^{2}||R_{K}||_{K}^{2}+
\sum\limits_{\sigma\in\varepsilon_{K}}D_{\sigma}^{2}h_{\sigma}||[\gamma_{{\bf
t}_{\sigma}}(S^{-1}{\bf u}_{h})]||_{\sigma}^{2}.
\end{equation*}
Here  $\alpha_{K}=\min(h_{K}/\sqrt{c_{S,K}},1/\sqrt{c_{{\bf w},r,K}})$,
$R_{K}$ is the elementwise residual
defined by
\begin{equation*}
R_{K} :=f-\nabla\cdot{\bf u}_{h}+(S^{-1}{\bf u}_{h})\cdot{\bf
w}-(r+\nabla\cdot{\bf w})p_{h},
\end{equation*}
and $D_{K}$, $D_\sigma$ denote two variants
of  coefficients over each element $K\in\mathcal{T}_{h}$ and each side
$\sigma\in\varepsilon_{h}$    respevtively given by
\begin{equation*}
D_{K}^{2}:=c_{{\bf w},r,K}+C_{{\bf w},r,K}^{2}\alpha_{K}^{2}+
\max_{K':\bar{K'}\cap\bar{K}\neq\emptyset}c_{{\bf w},r,K'}+\max_{K':\bar{K'}\cap\bar{K}\neq\emptyset}
\frac{|\nabla\cdot{\bf w}|_{K'}|^{2}}{c_{{\bf w},r,K'}},
\end{equation*}
\begin{equation*}
D_{\sigma}^{2} :=\frac{1}{2}\max_{K:\bar{K}\cap\bar{\sigma}\neq\emptyset}C_{S,K}+
\frac{1}{2}\min\{\max_{K:\bar{K}\cap\bar{\sigma}\neq\emptyset}\frac{C_{{\bf w},K}^{2}}{c_{{\bf w},r,K}},
\max_{K:\bar{K}\cap\bar{\sigma}\neq\emptyset}\frac{h_{K}^{2}C_{{\bf w},K}^{2}}{c_{S,K}}\}.
\end{equation*}

Define the global and local
errors, $e_{h}$ and $\mathcal{E}_{K}$,  of the stress and displacement variables as
\begin{equation}\label{error1}
e_{h}^{2} :=\sum\limits_{K\in\mathcal{T}_{h}}\mathcal{E}_{K}^{2},\quad
\mathcal{E}_{K}^{2} :=||S^{-1/2}({\bf u}-{\bf
u}_{h})||_{K}^{2}+c_{{\bf w},r,K}||p-p_{h}||_{K}^{2}.
\end{equation}
From \cite{D1}, or \cite{D2}  but with different forms of $D_{K}$  and $D_\sigma$, it holds an
upper bound estimate
\begin{equation}\label{equation39}
\displaystyle e_{h}^{2}\leq C_{1}\eta_{h}^{2} :=C_{1}\eta_{\mathcal{T}_{h}}^{2}({\bf u}_{h},p_{h},\mathcal{T}_{h})
:=C_{1}\sum\limits_{K\in\mathcal{T}_{h}}\eta_{\mathcal{T}_{h}}^{2}({\bf
u}_{h},p_{h}, K),
\end{equation}
where the positive constant $C_{1}$ depends only on the shape
regularity of the meshes.

For a given triangulation $\mathcal{T}_{h}$ and a
pair of corresponding discrete solutions $({\bf u}_{h},p_{h})\in
RT_{0}(\mathcal{T}_{h})\times P_{0}(\mathcal{T}_{h})$, we assume that  the module
ESTIMATE outputs the indicators
\begin{equation*}
\{\eta_{\mathcal{T}_{h}}^{2}({\bf
u}_{h},p_{h},K)\}_{K\in\mathcal{T}_{h}}={\rm ESTIMATE}({\bf
u}_{h},p_{h},\mathcal{T}_{h}).
\end{equation*}

Let $\bar{R}_{K}$ denote the mean of $R_{K}$ over each element
$K\in\mathcal{T}_{h}$. We define  the oscillation
\begin{equation}\label{oscillation1}
{\rm osc}_{h}^{2}
:=\displaystyle\sum\limits_{K\in\mathcal{T}_{h}}h_{K}^{2}||R_{K}-\bar{R}_{K}||_{K}^{2}.
\end{equation}

We note that throughout this paper  the triangulation $\mathcal{T}_{h}$
means a refinement of $\mathcal{T}_{H}$, and all notations with respect
to the mesh $\mathcal{T}_{H}$ are defined similarly.  We shall also use the notation $A\lesssim B$ to represent
$A\leq CB$ with $C>0$ a mesh-size independent, generic constant.

In MARK step, by   D\"{o}rfler marking we select the
elements to mark according to the indicators, namely, given a grid
$\mathcal{T}_{H}$ with the set of indicators
$\{\eta_{\mathcal{T}_{H}}^{2}({\bf
u}_{H},p_{H},K)\}_{K\in\mathcal{T}_{H}}$ and marking parameter
$\theta\in(0,1]$, the module MARK outputs a subset of making
elements, $\mathcal{M}_{H}\subset\mathcal{T}_{H}$, with
\begin{equation*}
\mathcal{M}_{H}={\rm MARK}(\{\eta_{\mathcal{T}_{H}}^{2}({\bf u}_{H},
p_{H},K)\}_{K\in \mathcal{T}_{H}},\mathcal{T}_{H},\theta)
\end{equation*}
satisfying D\"{o}rfler property
\begin{equation*}
\eta_{\mathcal{T}_{H}}({\bf u}_{H},p_{H},\mathcal{M}_{H}) :=\displaystyle(\sum\limits_{K\in\mathcal{M}_{H}}
\eta_{\mathcal{T}_{H}}^{2}({\bf u}_{H},p_{H},K))^{1/2}\geq
\theta\eta_{\mathcal{T}_{H}}({\bf u}_{H},p_{H},\mathcal{T}_{H}).
\end{equation*}

In REFINE step, we suppose that the refinement rule, such as the
longest edge bisection \cite{Rivara1,Rivara2} or the newest vertex
bisection \cite{Sewell,Mitchell1,Mitchell2}, is guaranteed to
produce conforming and shape regular meshes. Given a fixed integer
$b\geq1$, a mesh $\mathcal{T}_{H}$, and a subset
$\mathcal{M}_{H}\subset\mathcal{T}_{H}$ of marked elements, a
conforming triangulation $\mathcal{T}_{h}$ is output by
\begin{equation*}
\mathcal{T}_{h}={\rm REFINE}(\mathcal{T}_{H},\mathcal{M}_{H}),
\end{equation*}
where all elements of $\mathcal{M}_{H}$ are at least bisected $b$
times. Note that not only marked elements get refined but also
additional elements are refined to recover the conformity of
triangulations.

We now describe the AMFEM algorithm. In doing so, we replace the
subscript $H$ by an iteration counter called $k\geq0$. Let
$\mathcal{T}_{0}$ be a uniform triangulation with a marking
parameter $\theta\in(0,1]$. The basic loop of AMFEM is then given by
the following iterations:
\begin{center}
\begin{tabular}{|c|} \hline
AMFEM algorithm\\
Set $k =0$ and iterate\\
\hspace{-38mm}(1) $({\bf u}_{k},p_{k})={\rm SOLVE}(\mathcal{T}_{k})$; \\
\hspace{-1mm}(2)
$\{\eta_{k}^{2}({\bf u}_{k},p_{k},K)\}_{K\in\mathcal{T}_{k}}=
{\rm ESTIMATE}({\bf u}_{k},p_{k},\mathcal{T}_{k})$;\\
\hspace{-8mm}(3)
$\mathcal{M}_{k}={\rm MARK}(\{\eta_{k}^{2}({\bf u}_{k},
p_{k},K)\}_{K\in\mathcal{T}_{k}},\mathcal{T}_{k},\theta)$;\\
\hspace{-17.5mm}(4) $\mathcal{T}_{k+1}={\rm
REFINE}(\mathcal{T}_{k},\mathcal{M}_{k})$; $k=k+1$.\\ \hline
\end{tabular}
\end{center}
We note that the AMFEM algorithm is a standard one in which it
employs only the error estimator $\{\eta_{\mathcal{T}_{k}}^{2}({\bf
u}_{k},p_{k},K)\}_{K\in\mathcal{T}_{k}}$ and needs neither   marking the
oscillation nor the interior node property.

\section{Estimate for $L^{2}-$projection of the displacement}
This section is devoted to the estimation of
$||Q_{h}p-p_{h}||$, where $Q_{h}$ is the $L^{2}$-projection
operator onto $P_{0}(\mathcal{T}_{h})$. The estimate is one key to
the proof of convergence without the quasi-orthogonality available
due to the convection term. It gives as well
  a posteriori error estimates for the $L^{2}-$projection of the
displacement variable (see remark 3.1).

Consider the following
auxiliary problem:
\begin{equation}\label{equation4}
 \left \{ \begin{array}{ll}
    \nabla\cdot(S\nabla \phi)+\nabla\phi\cdot{\bf w}-
    r\phi  =Q_{h}p-p_{h} \quad   \mbox{in}\ \ \Omega,\\
    \hspace{39mm} \phi=0  \quad \mbox{on}\ \partial{\Omega}.
 \end{array}\right.
\end{equation}
It is well known that there exists a unique solution $\phi\in
H_{0}^{1}(\Omega)$ to the problem (\ref{equation4}) when the
convection and reaction terms satisfy $r+1/2\nabla\cdot{\bf
w}\geq0$ (Assumptions   ($D1$) and ($D4$)) withs the following regularity estimate:
\begin{equation}\label{equation5+}
||\phi||_{H^{1}}=||\phi||_{H^{1}(\bigcup\mathcal{T}_{h})}\lesssim
||Q_{h}p-p_{h}||.
\end{equation}
Moreover, if $\Omega$ is convex, $S\in C^{1,0}(\Omega)$ implies the estimate
\begin{equation}\label{equation5}
||\phi||_{H^{2}(\bigcup\mathcal{T}_{h})}\lesssim ||Q_{h}p-p_{h}||.
\end{equation}
We emphasize that we only need an estimate on $||\phi||_{H^{2}(K)}$
for each $K\in\mathcal{T}_{h}$, i.e., the assumption on $S$ could be
weaken in the sense that only (\ref{equation5}) is required. In
\cite{Carstensen0} Carstensen gave an example  which shows that
when $S$ is piecewise constant, $\phi$ satisfies (\ref{equation5})
but is not $H^{2}$-regular.

Set ${\bf z} :=S\nabla\phi\in H({\rm div},\Omega)$, and denote
$\phi_{h}$ the $L^{2}$-projection of $\phi$ onto
$P_{0}(\mathcal{T}_{h})$, and   $\Pi_{h}$ the interpolation
operator from $H({\rm div},\Omega)$ onto $RT_{0}(\mathcal {T}_{h})$
with the following estimate:
\begin{equation}\label{equation6}
||h^{-1}({\bf z}-\Pi_{h}{\bf z})||\lesssim|{\bf
z}|_{H^{1}(\bigcup\mathcal{T}_{h})}\ \ {\rm for\ all}\ \ {\bf z}\in
H({\rm div},\Omega).
\end{equation}
We refer to \cite{Arnold0,Brezzi;Fortin,Hiptmair} for the detailed
construction of such an interpolation operator $\Pi_{h}$ and the
approximation property.

From (\ref{equation-1}) and (\ref{equation2}), we obtain
\begin{equation}\label{equation7}
(Q_{h}p-p_{h},\nabla\cdot\Pi_{h}{\bf
z})=(p-p_{h},\nabla\cdot\Pi_{h}{\bf z})=(S^{-1}({\bf u}-{\bf
u}_{h}),\Pi_{h}{\bf z}).
\end{equation}
An integration by parts implies
\begin{equation}\label{equation8}
(S^{-1}({\bf u}-{\bf u}_{h}),{\bf z})=(S^{-1}({\bf u}-{\bf
u}_{h}),S\nabla\phi)=-(\nabla\cdot({\bf u}-{\bf u}_{h}),\phi).
\end{equation}
From (\ref{equation0}) and (\ref{equation3}) it follows
\begin{equation}\label{equation9}
\begin{array}{lll}
(\nabla\cdot({\bf u}-{\bf u}_{h}),\phi_{h})&=&(S^{-1}({\bf u}-{\bf
u}_{h})\cdot{\bf w},\phi_{h})-((r+\nabla\cdot{\bf
w})(p-p_{h}),\phi_{h})\vspace{2mm}\\
&=&(S^{-1}({\bf u}-{\bf u}_{h})\cdot{\bf
w},\phi_{h})-(p-p_{h},(r+\nabla\cdot{\bf
w})\phi_{h})\vspace{2mm}\\
&=&(S^{-1}({\bf u}-{\bf u}_{h})\cdot{\bf
w},\phi_{h})-(Q_{h}p-p_{h},(r+\nabla\cdot{\bf w})\phi_{h}).
\end{array}
\end{equation}

Denote $I :=(Q_{h}(\nabla\phi\cdot{\bf
w}),Q_{h}p-p_{h})-(r\phi_{h},Q_{h}p-p_{h})$. In view of the commuting
property of the interpolation operator $\Pi_{h}$, a combination of
(\ref{equation7})-(\ref{equation9}) yields
\begin{equation}\label{equation10}
\begin{array}{lll}
& &||Q_{h}p-p_{h}||^{2}=(Q_{h}p-p_{h},Q_{h}\nabla\cdot{\bf
z})+I=(Q_{h}p-p_{h},\nabla\cdot\Pi_{h}{\bf
z})+I\vspace{2mm}\\
& &\ =(S^{-1}({\bf u}-{\bf u}_{h}),\Pi_{h}{\bf z}-{\bf
z})+(S^{-1}({\bf u}-{\bf u}_{h}),{\bf
z})+I\vspace{2mm}\\
& &\ =(S^{-1}({\bf u}-{\bf u}_{h}),\Pi_{h}{\bf z}-{\bf
z})-(\nabla\cdot({\bf u}-{\bf
u}_{h}),\phi-\phi_{h})-(\nabla\cdot({\bf u}-{\bf
u}_{h}),\phi_{h})+I\vspace{2mm}\\
& &\ =(S^{-1}({\bf u}-{\bf u}_{h}),\Pi_{h}{\bf z}-{\bf
z})-(\nabla\cdot({\bf u}-{\bf u}_{h}),\phi-\phi_{h})-(S^{-1}({\bf
u}-{\bf
u}_{h})\cdot{\bf w},\phi_{h})\vspace{2mm}\\
& &\ \ \ \ \ +(\nabla\cdot{\bf
w}\phi_{h},Q_{h}p-p_{h})+(Q_{h}(\nabla\phi\cdot{\bf
w}),Q_{h}p-p_{h})\vspace{2mm}\\
& &\ =(S^{-1}({\bf u}-{\bf u}_{h}),\Pi_{h}{\bf z}-{\bf
z})-(\nabla\cdot({\bf u}-{\bf u}_{h}),\phi-\phi_{h})-(S^{-1}({\bf
u}-{\bf
u}_{h})\cdot{\bf w},\phi_{h}-\phi)\vspace{2mm}\\
& &\ \ \ \ \ -(S^{-1}({\bf u}-{\bf u}_{h})\cdot{\bf
w},\phi)+(\nabla\cdot{\bf
w}\phi_{h},Q_{h}p-p_{h})+(Q_{h}(\nabla\phi\cdot{\bf
w}),Q_{h}p-p_{h}).
\end{array}
\end{equation}

Recall the postprocessed technique developed by {\it Vohral\'{\i}k}
in \cite{Vohralik1}, where a
postprocessed approximation $\tilde{p}_{h}$ to the displacement $p$ is constructed
such that
$
-S_{K}\nabla\tilde{p}_{h}|_{K}={\bf u}_{h}$
and $
\frac{1}{|K|}\int_{K}\tilde{p}_{h}d{\bf x}=p_h|_{K}$  for all $ K\in\mathcal{T}_{h}.$
Then, from ${\bf w}\in
RT_{0}(\mathcal{T}_{h})$, we have
\begin{equation}\label{equation11}
\begin{array}{lll}
& &\displaystyle-(S^{-1}({\bf u}-{\bf u}_{h})\cdot{\bf
w},\phi)=\sum\limits_{K\in\mathcal{T}_{h}}\int_{K}\nabla(p-\tilde{p}_{h})\cdot{\bf
w}\phi\vspace{2mm}\\
& &\ \ \
\displaystyle=\sum\limits_{K\in\mathcal{T}_{h}}\int_{K}\nabla\cdot((p-\tilde{p}_{h}){\bf
w})\phi-\nabla\cdot{\bf w}(p-\tilde{p}_{h})\phi\vspace{2mm}\\
& &\ \ \
\displaystyle=\sum\limits_{K\in\mathcal{T}_{h}}\int_{K}-\nabla\phi\cdot{\bf
w}(p-\tilde{p}_{h})+\int_{\partial K}(p-\tilde{p}_{h}){\bf
w}\cdot{\bf n}\phi
-(\nabla\cdot{\bf w}\phi,p-\tilde{p}_{h})\vspace{2mm}\\
& &\ \ \ \displaystyle=-(\nabla\phi\cdot{\bf
w},p-\tilde{p}_{h})-(\nabla\cdot{\bf w}\phi,p-\tilde{p}_{h})-
\sum\limits_{\sigma\in\varepsilon_{h}^{0}}\int_{\sigma}[\tilde{p}_{h}]{\bf
w}\cdot{\bf n}\phi.
\end{array}
\end{equation}

Notice that it holds
\begin{equation}\label{equation12}
(Q_{h}(\nabla\phi\cdot{\bf
w}),Q_{h}p-p_{h})=(Q_{h}(\nabla\phi\cdot{\bf
w}),p-p_{h})=(Q_{h}(\nabla\phi\cdot{\bf w}),p-\tilde{p}_{h})
\end{equation}
and
\begin{equation}\label{equation13}
(\nabla\cdot{\bf w}\phi_{h},Q_{h}p-p_{h})=(\nabla\cdot{\bf
w}\phi_{h},p-p_{h})=(\nabla\cdot{\bf w}\phi_{h},p-\tilde{p}_{h}).
\end{equation}
For convenience, denote
\begin{equation*}
\begin{array}{lll}
& &I_{1} :=(S^{-1}({\bf u}-{\bf u}_{h}),\Pi_{h}{\bf z}-{\bf z}),\ \
I_{2} :=-(\nabla\cdot({\bf u}-{\bf
u}_{h}),\phi-\phi_{h}),\vspace{2mm}\\
& &I_{3} :=-(S^{-1}({\bf u}-{\bf u}_{h})\cdot{\bf
w},\phi_{h}-\phi),\ \ I_{4} :=-(\nabla\phi\cdot{\bf
w}-Q_{h}(\nabla\phi\cdot{\bf
w}),p-\tilde{p}_{h}),\vspace{2mm}\\
& &\displaystyle I_{5} :=-(\nabla\cdot{\bf
w}(\phi-\phi_{h}),p-\tilde{p}_{h}),\ \ I_{6}
:=-\sum\limits_{\sigma\in\varepsilon_{h}^{0}}\int_{\sigma}[\tilde{p}_{h}]{\bf
w}\cdot{\bf n}\phi.
\end{array}
\end{equation*}
From (\ref{equation10})-(\ref{equation13})  we arrive at
\begin{equation}\label{equation14}
\begin{array}{lll}
||Q_{h}p-p_{h}||^{2}&=&\displaystyle\sum\limits_{i=1}^{4}I_{i}-(\nabla\cdot{\bf
w}\phi,p-\tilde{p}_{h})+(\nabla\cdot{\bf
w}\phi_{h},Q_{h}p-p_{h})+I_{6}\vspace{2mm}\\
&=&\displaystyle\sum\limits_{i=1}^{4}I_{i}-(\nabla\cdot{\bf
w}(\phi-\phi_{h}),p-\tilde{p}_{h})-(\nabla\cdot{\bf
w}\phi_{h},p-\tilde{p}_{h})\vspace{2mm}\\
& &\ \ +(\nabla\cdot{\bf
w}\phi_{h},Q_{h}p-p_{h})+I_{6}=\displaystyle\sum\limits_{i=1}^{6}I_{i}.
\end{array}
\end{equation}

In what follows, we separately estimate $I_{i},i=1,...,6$.
\begin{lemma}\label{Lemma1}
Denote $||h||_{L^{\infty}(\Omega)}$ the maximum norm of the
mesh-size function $h$ with respect to $\mathcal{T}_{h}$, $e_{h}$
the error defined in (\ref{error1}). Then it holds
\begin{equation}\label{equation15}
I_{1}\lesssim ||h||_{L^{\infty}(\Omega)}\ e_{h}||Q_{h}p-p_{h}||.
\end{equation}
\end{lemma}
\begin{proof}
From (\ref{equation6}) and (\ref{equation5}) it follows
\begin{equation}\label{equation16}
\begin{array}{lll}
I_{1}&=&\displaystyle\sum\limits_{K\in\mathcal{T}_{h}}(S^{-1}({\bf
u}-{\bf u}_{h}),\Pi_{h}{\bf z}-{\bf
z})_{K}\vspace{2mm}\\
&\lesssim&\displaystyle\sum\limits_{K\in\mathcal{T}_{h}}||S^{-1/2}({\bf
u}-{\bf u}_{h})||_{K}||\Pi_{h}{\bf z}-{\bf z}||_{K}\vspace{2mm}\\
&\lesssim&||h||_{L^{\infty}(\Omega)}e_{h}|{\bf
z}|_{H^{1}(\bigcup\mathcal{T}_{h})}\vspace{2mm}\\
&\lesssim&  ||h||_{L^{\infty}(\Omega)}\ e_{h} ||Q_{h}p-p_{h}||.
\end{array}
\end{equation}
\end{proof}
\begin{lemma}\label{Lemma2} It holds
\begin{equation}\label{equation17}
I_{2}\lesssim(||h||_{L^{\infty}(\Omega)}e_{h}+{\rm
osc}_{h})||Q_{h}p-p_{h}||.
\end{equation}
\end{lemma}
\begin{proof}
Notice that (\ref{equation0}) can be equivalently written as:
\begin{equation}\label{equation18}
(\nabla\cdot{\bf u},\varphi)_{K}-(S^{-1}{\bf u}\cdot{\bf
w},\varphi)_{K}+((r+\nabla\cdot{\bf
w})p,\varphi)_{K}=(f,\varphi)_{K}\ \ \ {\rm for\ all}\ \ \varphi\in
L^{2}(K), K\in\mathcal{T}_{h}.
\end{equation}
Meanwhile, the relation
(\ref{equation3}) can be equivalently written as:
\begin{equation}\label{equation19}
(\nabla\cdot{\bf u}_{h},\varphi)_{K}-(S^{-1}{\bf u}_{h}\cdot{\bf
w},\varphi)_{K}+((r+\nabla\cdot{\bf
w})p_{h},\varphi)_{K}=(f,\varphi)_{K}\ \ \ {\rm for\ all}\ \
\varphi\in P_{0}(K), K\in\mathcal{T}_{h}.
\end{equation}

 For arbitrary $\varphi\in
L^{2}(K)$, let $\bar\varphi_{K}$ denote the mean of $\varphi$ over
$K\in\mathcal{T}_{h}$. A combination of (\ref{equation18}) and
(\ref{equation19}) yields
\begin{equation*}
\begin{array}{lll}
&\ &(\nabla\cdot({\bf u}-{\bf u}_{h}),\varphi)_{K}=(\nabla\cdot{\bf
u},\varphi)_{K}-(\nabla\cdot{\bf
u}_{h},\varphi-\bar\varphi_{K})_{K}-(\nabla\cdot{\bf
u}_{h},\bar\varphi_{K})_{K}\vspace{2mm}\\
&=&(R_{K},\varphi-\bar\varphi_{K})_{K}+(S^{-1}({\bf u}-{\bf
u}_{h})\cdot{\bf w},\varphi)_{K}-((r+\nabla\cdot{\bf
w})(p-p_{h}),\varphi)_{K}\vspace{2mm}\\
&=&(R_{K}-\bar{R}_{K},\varphi-\bar\varphi_{K})_{K}+(S^{-1}({\bf u}-{\bf
u}_{h})\cdot{\bf w},\varphi)_{K}-((r+\nabla\cdot{\bf
w})(p-p_{h}),\varphi)_{K}\vspace{2mm}\\
&\leq&(||R_{K}-\bar{R}_{K}||_{K}+||S^{-1}({\bf u}-{\bf
u}_{h})||_{K}||{\bf w}||_{L^{\infty}(K)}+C_{{\bf
w},r,K}||p-p_{h}||_{K})||\varphi||_{K}.
\end{array}
\end{equation*}
Here we recall that $\bar{R}_{K}$ is the mean   of the elementwise residual
$R_{K}$ over each $K\in\mathcal{T}_{h}$.
The above inequality indicates
\begin{equation}\label{equation20}
\begin{array}{lll}
||\nabla\cdot({\bf u}-{\bf u}_{h})||_{K}&=&\displaystyle\sup_{\varphi\in
L^{2}(K),\varphi\neq0}\frac{(\nabla\cdot({\bf u}-{\bf
u}_{h}),\varphi)_{K}}{||\varphi||_{L^{2}(K)}}\vspace{2mm}\\
&\lesssim&||R_{K}-\bar{R}_{K}||_{K}+\mathcal{E}_{K}.
\end{array}
\end{equation}
Then it follows
\begin{equation}\label{equation21}
\begin{array}{lll}
I_{2}&=&\displaystyle-\sum\limits_{K\in\mathcal{T}_{h}}(\nabla\cdot({\bf
u}-{\bf
u}_{h}),\phi-\phi_{h})\vspace{2mm}\\
&\lesssim&\displaystyle\sum\limits_{K\in\mathcal{T}_{h}}||\nabla\cdot({\bf
u}-{\bf u}_{h})||_{K}h_{K}||\nabla\phi||_{K}\vspace{2mm}\\
&\lesssim&(||h||_{L^{\infty}(\Omega)}\ e_{h}+{\rm
osc}_{h})||\nabla\phi||.
\end{array}
\end{equation}
The desired result (\ref{equation17}) follows from
(\ref{equation21}) and (\ref{equation5+}).
\end{proof}
\begin{lemma}\label{Lemma3} It holds
\begin{equation}\label{equation22}
I_{3}\lesssim ||h||_{L^{\infty}(\Omega)}\ e_{h}||Q_{h}p-p_{h}||.
\end{equation}
\end{lemma}
\begin{proof}
By noticing
\begin{equation}\label{equation23}
\begin{array}{lll}
I_{3}&=&\displaystyle-\sum\limits_{K\in\mathcal{T}_{h}}(S^{-1}({\bf
u}-{\bf u}_{h})\cdot{\bf w},\phi_{h}-\phi)_{K}\vspace{2mm}\\
&\lesssim&\displaystyle\sum\limits_{K\in\mathcal{T}_{h}}||S^{-1/2}({\bf
u}-{\bf u}_{h})||_{K}h_{K}||\nabla\phi||_{K}\vspace{2mm}\\
&\leq&||h||_{L^{\infty}(\Omega)}\ e_{h}||\phi||_{H^{1}}
\end{array}
\end{equation}
the desired result (\ref{equation22}) follows from
(\ref{equation23}) and  the regularity estimate (\ref{equation5+}).
\end{proof}
\begin{lemma}\label{Lemma4} It holds
\begin{equation}\label{equation24}
I_{4}\lesssim ||h||_{L^{\infty}(\Omega)}\ e_{h}||Q_{h}p-p_{h}||.
\end{equation}
\end{lemma}
\begin{proof}
Recall a local efficiency estimate of $h_{K}||S^{-1}{\bf
u}_{h}||_{K}$ as following (see Lemma 7.3 in \cite{D1,D2}):
\begin{equation}\label{equation25}
h_{K}||S^{-1}{\bf u}_{h}||_{K}\lesssim\mathcal{E}_{K}.
\end{equation}

By   triangle
inequality we obtain
\begin{equation}\label{equation26}
\begin{array}{lll}
I_{4}&=&\displaystyle-\sum\limits_{K\in\mathcal{T}_{h}}(\nabla\phi\cdot{\bf
w}-Q_{h}(\nabla\phi\cdot{\bf
w}),p-\tilde{p}_{h})_{K}\vspace{2mm}\\
&\lesssim&\displaystyle\sum\limits_{K\in\mathcal{T}_{h}}h_{K}||\nabla(\nabla\phi\cdot{\bf
w})||_{K}(||p-p_{h}||_{K}+||p_{h}-\tilde{p}_{h}||_{K})_{K}\vspace{2mm}\\
&\lesssim&\displaystyle\sum\limits_{K\in\mathcal{T}_{h}}h_{K}||\phi||_{H^{2}(K)}(||p-p_{h}||_{K}
+h_{K}||S^{-1}{\bf u}_{h}||_{K}).
\end{array}
\end{equation}
From (\ref{equation25}) it holds
\begin{equation}\label{equation27}
||p-p_{h}||_{K}+h_{K}||S^{-1}{\bf
u}_{h}||_{K}\lesssim\mathcal{E}_{K}.
\end{equation}

According to (\ref{equation26}) and (\ref{equation27}), we arrive at
\begin{equation*}
\displaystyle
I_{4}\lesssim\sum\limits_{K\in\mathcal{T}_{h}}h_{K}||\phi||_{H^{2}(K)}\mathcal{E}_{K},
\end{equation*}
which, together with  the regularity estimate (\ref{equation5}) of $\phi$, implies the  result (\ref{equation24}).

\end{proof}
\begin{lemma}\label{Lemma5}
It holds
\begin{equation}\label{equation28}
I_{5}\lesssim e_{h} ||h||_{L^{\infty}(\Omega)}||Q_{h}p-p_{h}||.
\end{equation}
\end{lemma}
\begin{proof}
From (\ref{equation27}) it follows
\begin{equation*}
\begin{array}{lll}
I_{5}&=&\displaystyle-\sum\limits_{K\in\mathcal{T}_{h}}(\nabla\cdot{\bf
w}(\phi-\phi_{h}),p-\tilde{p}_{h})_{K}\vspace{2mm}\\
&\lesssim&\displaystyle\sum\limits_{K\in\mathcal{T}_{h}}||\phi-\phi_{h}||_{K}(||p-p_{h}||_{K}+||p_{h}-\tilde{p}_{h}||_{K})\vspace{2mm}\\
&\lesssim&\displaystyle\sum\limits_{K\in\mathcal{T}_{h}}h_{K}||\nabla\phi||_{K}(||p-p_{h}||_{K}+h_{K}||S^{-1}{\bf
u}_{h}||_{K})\vspace{2mm}\\
&\lesssim&\displaystyle e_{h}||h||_{L^{\infty}(\Omega)}||\phi||_{H^{1}},
\end{array}
\end{equation*}
which,  together with  (\ref{equation5+}),  yields the estimate (\ref{equation28}).
\end{proof}
\begin{lemma}\label{Lemma6}
It holds
\begin{equation}\label{equation29}
I_{6}\lesssim e_{h}||h||_{L^{\infty}(\Omega)}^{1/2}||Q_{h}p-p_{h}||.
\end{equation}
\end{lemma}
\begin{proof}
For any $\sigma\in\varepsilon_{h}^{0}$, let $\phi_{\sigma}$ denote
the mean of $\phi$ onto $\sigma$, i.e., $\phi_{\sigma}
:=\displaystyle\frac{1}{|\sigma|}\int_{\sigma}\phi ds$. According to the
continuity of the means of traces of the postprocessed scalar
$\tilde{p}_{h}$ (see Lemma 6.1 in \cite{Vohralik1}), and noticing
${\bf w}\in RT_{0}(\mathcal{T}_{h})$, we have
\begin{equation}\label{equation30}
\begin{array}{lll}
\displaystyle\int_{\sigma}[\tilde{p}_{h}]{\bf w}\cdot{\bf
n}\phi&=&\displaystyle\int_{\sigma}[\tilde{p}_{h}]{\bf w}\cdot{\bf
n}(\phi-\phi_{\sigma})\vspace{2mm}\\
&\lesssim&\displaystyle||[\tilde{p}_{h}]||_{\sigma}||\phi-\phi_{\sigma}||_{\sigma}.
\end{array}
\end{equation}

A sidewise Poincar\'{e} inequality and trace theory imply
\begin{equation}\label{equation31}
||\phi-\phi_{\sigma}||_{\sigma}\lesssim
h_{\sigma}||\gamma_{{\bf t}_{\sigma}}(\nabla\phi)||_{\sigma}\lesssim
h_{\sigma}||\phi||_{H^{2}(\bigcup\omega_{\sigma})}.
\end{equation}

From trace theorem, generalized Friedrichs inequality (see (2.2) in
\cite{Vohralik1}), and the postprocessed technique, we have
\begin{equation}\label{equation32}
\displaystyle||[\tilde{p}_{h}]||_{\sigma}=||[\tilde{p}_{h}]-\frac{1}{|\sigma|}\int_{\sigma}[\tilde{p}_{h}]ds||_{\sigma}
\lesssim h_{\sigma}^{1/2}||\nabla_{h}\tilde{p}_{h}||_{\omega_{\sigma}}\leq
h_{\sigma}^{1/2}||S^{-1}{\bf u}_{h}||_{\omega_{\sigma}}.
\end{equation}
A combination of (\ref{equation30})-(\ref{equation32}) yields
\begin{equation}\label{equation33}
\displaystyle\int_{\sigma}[\tilde{p}_{h}]{\bf w}\cdot{\bf n}\phi\lesssim
h_{\sigma}||S^{-1}{\bf
u}_{h}||_{\omega_{\sigma}}h_{\sigma}^{1/2}||\phi||_{H^{2}(\bigcup\omega_{\sigma})}.
\end{equation}
In light of  the local shape regularity
of element, the above estimate leads to
\begin{equation}\label{equation34}
\begin{array}{lll}
I_{6}&\lesssim&\displaystyle\sum\limits_{\sigma\in\varepsilon_{h}^{0}}h_{\sigma}||S^{-1}{\bf
u}_{h}||_{\omega_{\sigma}}h_{\sigma}^{1/2}||\phi||_{H^{2}(\bigcup\omega_{\sigma})}\vspace{2mm}\\
&\lesssim&\displaystyle||h||_{L^{\infty}(\Omega)}^{1/2}(\sum\limits_{K\in\mathcal{T}_{h}}h_{K}^{2}||S^{-1}{\bf
u}_{h}||_{K}^{2})^{1/2}||\phi||_{H^{2}(\bigcup\mathcal{T}_{h})}.
\end{array}
\end{equation}
The desired result (\ref{equation29}) follows from
(\ref{equation34}),   (\ref{equation25}) and
(\ref{equation5}).
\end{proof}

We now give an estimate of $||Q_{h}p-p_{h}||$.
\begin{theorem}\label{theorem1}
Let $e_{h}$, ${\rm osc}_{h}$, and $||h||_{L^{\infty}(\Omega)}$
denote the error of the stress and displacement variables given in (\ref{error1}), the
oscillation of data given in (\ref{oscillation1}) , and the maximum norm of the mesh-size function,
respectively, with respect to $\mathcal{T}_{h}$. Then there exits a
positive constant $C_{2}$ only depending on the shape regularity of
$\mathcal{T}_{h}$, such that
\begin{equation}\label{equation35}
||Q_{h}p-p_{h}||\leq
C_{2}C_{D}(||h||_{L^{\infty}(\Omega)}^{1/2}e_{h}+{\rm osc}_{h}),
\end{equation}
where $C_{D}$ is one variant of coefficients.
\end{theorem}
\begin{proof}
The estimate (\ref{equation35}) follows from a combination of
(\ref{equation14}), (\ref{equation15}), (\ref{equation17}),
(\ref{equation22}), (\ref{equation24}), (\ref{equation28}), and
(\ref{equation29}).
\end{proof}

\begin{remark} {\it A combination of the   two estimates (\ref{equation35}) and
(\ref{equation39}) actually give a posteriori bound for
$Q_{h}(p-p_{h})$. Furthermore,  following \cite{Brezzi;Fortin}, we denote
\begin{equation*}
\displaystyle\mathcal{L}_{k}^{1,NC} :=\{q_{h}\in L^{2}(\Omega):
q_{h}|_{K}\in P_{k}(K), \forall
K\in\mathcal{T}_{h},\sum\limits_{K}\int_{\partial K}p_{h}\phi
ds=0,\forall\phi\in\mathcal{R}_{k}(\partial K)\},
\end{equation*}
and let $p_{h}^{*}, \hat{p}_{h}$ be respectively the interpolates in $\mathcal{L}_{k}^{1,NC}$ of the interelement Lagrangian multiplier  $\lambda_h$ and the displacement $p$ (\cite{Brezzi;Fortin}, pages 186-187; We note that in \cite{Brezzi;Fortin} $ u$ represents the displacement variable and $p$ the stress variable). Following the same line as in \cite{Brezzi;Fortin}, it holds the estimate
\begin{equation*}
||\hat{p}_{h}-p_{h}^{*}||\lesssim||hS^{-1/2}({\bf u}-{\bf
u}_{h})||+||Q_{h}(p-p_{h})||,
\end{equation*}
which gives an a posteriori error estimate for
$\hat{p}_{h}-p_{h}^{*}$.}
\end{remark}
\begin{remark} For a pure diffusion problem, i.e., ${\bf w}=r=0$ in (\ref{equation1}),
it holds $I_{i}=0,i=3,4,5,6$. From the estimates of $I_{1}$ and
$I_{2}$, we can obtain
\begin{equation*}
||Q_{h}p-p_{h}||\lesssim||hS^{-1/2}({\bf u}-{\bf u}_{h})||+{\rm
osc}_{h},
\end{equation*}
which results in the quasi-orthogonality
\begin{equation*}
(S^{-1}({\bf u}-{\bf u}_{h}),{\bf u}_{h}-{\bf
u}_{H})\lesssim(||hS^{-1/2}({\bf u}-{\bf u}_{h})||+{\rm
osc}_{h})||f_{h}-f_{H}||,
\end{equation*}
where we have used the fact that $\nabla\cdot{\bf
u}_{h}=Q_{h}f=f_{h}$ and $\nabla\cdot{\bf u}_{H}=Q_{H}f=f_{H}$. This
estimate is somewhat different from the quasi-orthogonality results
in \cite{Carstensen;Hoppe,Carstensen;Rabus,Chen;Holst;Xu,Du2,Du3}.
\end{remark}

\section{Estimator reduction} 
Let $\omega_{\sigma}$ denote
the patch of $\sigma\in\varepsilon_{h}$, and define
$c_{\omega_{\sigma}}$, $D_{\omega_{K}}^{2}$, $D_{\mathcal{T}_{h}}^{2}(K)$, $D_{\mathcal{T}_{h}}^{2}$ respectively by
\begin{equation*}
c_{\omega_{\sigma}} :=\left \{ \begin{array}{ll}
  \max(c_{S,K}^{-1/2},c_{S,L}^{-1/2})\ \  & \mbox{if}\;  \ \ \sigma=K\cap L,\\
 c_{S,K}^{-1/2}\ \ \ \  & \mbox{if}\; \ \
 \sigma\in\varepsilon_{K}\cap\partial\Omega,
 \end{array}\right.\quad
D_{\omega_{K}}^{2}
:=\max_{\sigma\in\varepsilon_{K}}D_{\sigma}^{2}c_{\omega_{\sigma}}^{2},
\end{equation*}
$$ D_{\mathcal{T}_{h}}^{2}(K)
:=\max(h_{K}^{2}D_{K}^{2}c_{S,K}^{-1},C_{DK}\alpha_{K}^{2},D_{\omega_{K}}^{2}),
\quad D_{\mathcal{T}_{h}}^{2}
:=\displaystyle\sum\limits_{K\in\mathcal{T}_{h}}D_{\mathcal{T}_{h}}^{2}(K),$$
where   $C_{DK}$ in $ D_{\mathcal{T}_{h}}^{2}(K)$ is given by
$$C_{DK}:=2\max((C_{4}C_{S,K}^{1/2}+\frac{h_{K}}{\sqrt{c_{S,K}}}||{\bf
w}||_{L^{\infty}(K)})^{2},\frac{h_{K}^{2}C_{{\bf
w},r,K}^{2}}{c_{{\bf w},r,K}}).$$

\begin{lemma}\label{lemma7} {\rm(Estimator reduction)}\ {\it For a triangulation $\mathcal{T}_{H}$
with $\mathcal{M}_{H}\subset\mathcal{T}_{H}$, let $\mathcal{T}_{h}$
be a refinement of $\mathcal{T}_{H}$ obtained by the AMFEM
algorithm. Denote by $D_{\mathcal{T}_{0}}^{2}$ one variant of the
coefficients onto the initial mesh $\mathcal{T}_{0}$, and denote $\lambda
:=1-2^{-b/2}$,
\begin{equation}\label{estimator1}
E_{H}^{2}
:=\displaystyle\sum\limits_{K\in\mathcal{T}_{h}}||S^{-1/2}({\bf
u}_{h}-{\bf u}_{H})||_{K}^{2}+c_{{\bf
w},r,K}||p_{h}-p_{H}||_{K}^{2}.
\end{equation}
Then for any $0<\delta<1$, it holds
\begin{equation}\label{equation40}
\eta_{\mathcal{T}_{h}}^{2}({\bf
u}_{h},p_{h},\mathcal{T}_{h})\leq(1+\delta)\{\eta_{\mathcal{T}_{H}}^{2}({\bf
u}_{H},p_{H},\mathcal{T}_{H})-\lambda\eta_{\mathcal{T}_{H}}^{2}({\bf
u}_{H},p_{H},\mathcal{M}_{H})\}+C_{3}^{2}(1+\delta^{-1})D_{\mathcal{T}_{0}}^{2}E_{H}^{2}.
\end{equation}
}
\end{lemma}
\begin{proof} By   triangle inequality and Young inequality, we have
\begin{equation}\label{equation41}
h_{K}^{2}||S^{-1}{\bf
u}_{h}||_{K}^{2}\leq(1+\delta)h_{K}^{2}||S^{-1}{\bf
u}_{H}||_{K}^{2}+(1+\delta^{-1})h_{K}^{2}c_{S,K}^{-1}||S^{-1/2}({\bf
u}_{h}-{\bf u}_{H})||_{K}^{2}.
\end{equation}
Inverse inequality implies
\begin{equation*}
||\nabla\cdot({\bf u}_{h}-{\bf u}_{H})||_{K}\leq
C_{4}C_{S,K}^{1/2}h_{K}^{-1}||S^{-1/2}({\bf u}_{h}-{\bf
u}_{H})||_{K},
\end{equation*}
which leads to
\begin{equation}\label{equation42}
\begin{array}{lll}
& &h_{K}^{2}||f-\nabla\cdot{\bf u}_{h}+S^{-1}{\bf u}_{h}\cdot{\bf
w}-(r+\nabla\cdot{\bf w})p_{h}||_{K}^{2}\vspace{2mm}\\
& &\ \leq(1+\delta)h_{K}^{2}||f-\nabla\cdot{\bf u}_{H}+S^{-1}{\bf
u}_{H}\cdot{\bf w}-(r+\nabla\cdot{\bf
w})p_{H}||_{K}^{2}\vspace{2mm}\\
& &\ \ \ +(1+\delta^{-1})C_{DK}(||S^{-1/2}({\bf u}_{h}-{\bf
u}_{H})||_{K}^{2}+c_{{\bf w},r,K}||p_{h}-p_{H}||_{K}^{2}).
\end{array}
\end{equation}

From inverse inequality and local shape regularity of the mesh, it follows
\begin{equation}\label{equation43}
||[\gamma_{{\bf t}_{\sigma}}(S^{-1}{\bf
u}_{h})]||_{\sigma}\leq||[\gamma_{{\bf t}_{\sigma}}(S^{-1}{\bf
u}_{H})]||_{\sigma}+C_{5}c_{\omega_{\sigma}}h_{\sigma}^{-1/2}||S^{-1/2}({\bf
u}_{h}-{\bf u}_{H})||_{\omega_{\sigma}},
\end{equation}
which, together with Young inequality, yields
\begin{equation}\label{equation44}
h_{\sigma}||[\gamma_{{\bf t}_{\sigma}}(S^{-1}{\bf
u}_{h})]||_{\sigma}^{2}\leq (1+\delta)h_{\sigma}||[\gamma_{{\bf
t}_{\sigma}}(S^{-1}{\bf
u}_{H})]||_{\sigma}^{2}+(1+\delta^{-1})C_{5}^{2}c_{\omega_{\sigma}}^{2}||S^{-1/2}({\bf
u}_{h}-{\bf u}_{H})||_{\omega_{\sigma}}^{2}.
\end{equation}

Summing (\ref{equation41}) and (\ref{equation42}) over all elements
$K\in\mathcal{T}_{h}$,   summing (\ref{equation44}) over all interior sides
$\sigma\in\varepsilon_{h}^{0}$, and noticing the monotonicity of
variants of coefficients, we obtain
\begin{equation}\label{equation45}
\eta_{\mathcal{T}_{h}}^{2}({\bf
u}_{h},p_{h},\mathcal{T}_{h})\leq(1+\delta)\eta_{\mathcal{T}_{h}}^{2}({\bf
u}_{H},p_{H},\mathcal{T}_{h})+C_{3}^{2}(1+\delta^{-1})D_{\mathcal{T}_{h}}^{2}E_{H}^{2}.
\end{equation}
For a marked element $K\in\mathcal{M}_{H}$, we set
$\mathcal{T}_{h,K} :=\{K'\in\mathcal{T}_{h}|K'\subset K\}$. It holds
\begin{equation*}
\left \{ \begin{array}{ll}
  \displaystyle\sum\limits_{K'\in\mathcal{T}_{h,K}}\eta_{\mathcal{T}_{h}}^{2}({\bf
u}_{H},p_{H},K')\leq2^{-b/2}\eta_{\mathcal{T}_{H}}^{2}({\bf
u}_{H},p_{H},K)\ \  & \mbox{for}\;  \ \ \ K\in\mathcal{T}_{H}/\mathcal{T}_{h},\\
 \eta_{\mathcal{T}_{h}}^{2}({\bf
u}_{H},p_{H},K)\leq\eta_{\mathcal{T}_{H}}^{2}({\bf u}_{H},p_{H},K)\
\ \ \  & \mbox{for}\; \ \
 K\in\mathcal{T}_{H}/\mathcal{M}_{H},
 \end{array}\right.
\end{equation*}
which results in the following estimate
\begin{equation}\label{equation46}
\begin{array}{lll}
\eta_{\mathcal{T}_{h}}^{2}({\bf
u}_{H},p_{H},\mathcal{T}_{h})&\leq&2^{-b/2}\eta_{\mathcal{T}_{H}}^{2}({\bf
u}_{H},p_{H},\mathcal{M}_{H})+\eta_{\mathcal{T}_{H}}^{2}({\bf
u}_{H},p_{H},\mathcal{T}_{H}/\mathcal{T}_{h})\vspace{2mm}\\
&=&\eta_{\mathcal{T}_{H}}^{2}({\bf
u}_{H},p_{H},\mathcal{T}_{H})-\lambda\eta_{\mathcal{T}_{H}}^{2}({\bf
u}_{H},p_{H},\mathcal{M}_{H}).
\end{array}
\end{equation}
The desired result (\ref{equation40}) follows from
(\ref{equation45}), (\ref{equation46}) and the
monotonicity of $D_{\mathcal{T}_{h}}$.
\end{proof}

\section{Proof of theorem \ref{theorem--}}.\ In this section, we show that the
error plus some quantity uniformly reduces with a fixed factor on
two successive meshes, which shows AMFEM is convergent.
\begin{lemma}\label{Lemma8}
Let $({\bf u}_{h},p_{h})\in RT_{0}(\mathcal{T}_{h})\times
P_{0}(\mathcal{T}_{h})$ be the approximation solutions to the
stress and   displacement variables with respect to
$\mathcal{T}_{h}$, and $e_{h}$   the error of the stress and
displacement variables with respect to $\mathcal{T}_{h}$. Denote by
$h_{0}$ the mesh-size of the quasi-uniform initial mesh
$\mathcal{T}_{0}$, by $D_{1}$ one variant of the coefficients
determined by $C_{D}$. Then it holds for any $0<\delta_{1}<1$
\begin{equation}\label{equation48}
\begin{array}{lll}
& &||S^{-1/2}({\bf u}-{\bf
u}_{h})||^{2}\leq2\delta_{1}^{-1}D_{1}(h_{0}e_{h}^{2}+h_{0}^{2}(||f-f_{h}||^{2}+||h\nabla_{h}(S^{-1}{\bf
u}_{h}\cdot{\bf w})||^{2}))\vspace{2mm}\\
& &\ \ \ \ \ \hspace{0.5cm}+||S^{-1/2}({\bf u}-{\bf
u}_{H})||^{2}-||S^{-1/2}({\bf u}_{h}-{\bf
u}_{H})||^{2}+1/2\delta_{1}||\nabla\cdot({\bf u}_{h}-{\bf
u}_{H})||^{2}.
\end{array}
\end{equation}
\end{lemma}
\begin{proof}
From (\ref{equation-1}) and (\ref{equation2})  we get
\begin{equation}\label{equation49}
\begin{array}{lll}
||S^{-1/2}({\bf u}-{\bf u}_{h})||^{2}&=&||S^{-1/2}({\bf u}-{\bf
u}_{H})||^{2}-||S^{-1/2}({\bf
u}_{h}-{\bf u}_{H})||^{2}\vspace{2mm}\\
& &-2(S^{-1}({\bf u}-{\bf u}_{h}),{\bf u}_{h}-{\bf
u}_{H})\vspace{2mm}\\
&=&||S^{-1/2}({\bf u}-{\bf u}_{H})||^{2}-||S^{-1/2}({\bf
u}_{h}-{\bf u}_{H})||^{2}\vspace{2mm}\\
& &-2(p-p_{h},\nabla\cdot({\bf u}_{h}-{\bf u}_{H}))\vspace{2mm}\\
&=&||S^{-1/2}({\bf u}-{\bf u}_{H})||^{2}-||S^{-1/2}({\bf
u}_{h}-{\bf u}_{H})||^{2}\vspace{2mm}\\
& &-2(Q_{h}p-p_{h},\nabla\cdot({\bf u}_{h}-{\bf u}_{H})).
\end{array}
\end{equation}
The definition of the residual $R_{K}$ and the assumptions of data
imply that for each $K\in\mathcal{T}_{h}$,
\begin{equation*}
h_{K}||R_{K}-\bar{R}_{K}||_{K}\leq
C_{6}(h_{K}^{2}||\nabla(S^{-1}{\bf u}_{h}\cdot{\bf
w})||_{K}+h_{K}||f-f_{h}||_{K}),
\end{equation*}
which, together with the fact
$||h||_{L^{\infty}(\Omega)}\leq h_{0}$ and the definition (\ref{oscillation1}) of ${\rm
osc}_{h}$, yields the estimate
\begin{equation}\label{equation50}
{\rm osc}_{h}^{2}\leq
2C_{6}^{2}h_{0}^{2}(||f-f_{h}||^{2}+||h\nabla_{h}(S^{-1}{\bf
u}_{h}\cdot{\bf w})||^{2}).
\end{equation}

Applying the above estimate (\ref{equation50}) to
(\ref{equation35}), we obtain
\begin{equation}\label{equation51}
||Q_{h}p-p_{h}||^{2}\leq
D_{1}(h_{0}e_{h}^{2}+h_{0}^{2}(||f-f_{h}||^{2}+||h\nabla_{h}(S^{-1}{\bf
u}_{h}\cdot{\bf w})||^{2})).
\end{equation}
In light of  Young inequality, we have
\begin{equation}\label{equation52}
2(Q_{h}p-p_{h},\nabla\cdot({\bf u}_{h}-{\bf
u}_{H}))\leq2\delta_{1}^{-1}||Q_{h}p-p_{h}||^{2}+\delta_{1}/2||\nabla\cdot({\bf
u}_{h}-{\bf u}_{H})||^{2}.
\end{equation}
The desired result (\ref{equation48}) follows from a combination of
(\ref{equation49}), (\ref{equation51}), and (\ref{equation52}).
\end{proof}
\begin{lemma}{\label{Lemma9}}\ Let $D_{2}$ and $D_{3}$
be two variants of the coefficients respectively given by
\begin{equation*}
D_{2} :=\max_{K\in\mathcal{T}_{h}}||{\bf
w}||_{L^{\infty}(K)}^{2}c_{S,K}^{-1},\ \ D_{3}
:=\max_{K\in\mathcal{T}_{h}}C_{{\bf w},r,K}.
\end{equation*}
Under the assumptions of Lemma \ref{Lemma8}, it holds
\begin{equation}\label{equation53}
\begin{array}{lll}
& &||\nabla\cdot({\bf u}-{\bf
u}_{h})||^{2}\leq4D_{1}D_{3}(h_{0}e_{h}^{2}+h_{0}^{2}(||f-f_{h}||^{2}+||h\nabla_{h}(S^{-1}{\bf
u}_{h}\cdot{\bf w})||^{2}))\vspace{2mm}\\
& &\ \hspace{0.5cm}+||\nabla\cdot({\bf u}-{\bf
u}_{H})||^{2}-1/2||\nabla\cdot({\bf u}_{h}-{\bf
u}_{H})||^{2}+4D_{2}||S^{-1/2}({\bf u}-{\bf u}_{h})||^{2}.
\end{array}
\end{equation}
\end{lemma}
\begin{proof}\ Notice
\begin{equation}\label{equation54}
||\nabla\cdot({\bf u}-{\bf u}_{h})||^{2}=||\nabla\cdot({\bf u}-{\bf
u}_{H})||^{2}-||\nabla\cdot({\bf u}_{h}-{\bf
u}_{H})||^{2}-2(\nabla\cdot({\bf u}-{\bf u}_{h}),\nabla\cdot({\bf
u}_{h}-{\bf u}_{H})).
\end{equation}
A combination of (\ref{equation0}) and (\ref{equation3}) leads to
\begin{equation}\label{equation55}
\begin{array}{lll}
& &(\nabla\cdot({\bf u}-{\bf u}_{h}),\nabla\cdot({\bf u}_{h}-{\bf
u}_{H}))=(S^{-1}({\bf u}-{\bf u}_{h})\cdot{\bf w},\nabla\cdot({\bf
u}_{h}-{\bf u}_{H}))\vspace{2mm}\\
& &\ \hspace{4.4cm}-((r+\nabla\cdot{\bf
w})(Q_{h}p-p_{h}),\nabla\cdot({\bf u}_{h}-{\bf u}_{H}))\vspace{2mm}\\
& &\leq2D_{2}||S^{-1/2}({\bf u}-{\bf
u}_{h})||^{2}+2D_{3}||Q_{h}p-p_{h}||^{2}+1/4||\nabla\cdot({\bf
u}_{h}-{\bf u}_{H})||.
\end{array}
\end{equation}
The estimate (\ref{equation53}) follows from
(\ref{equation54}), (\ref{equation51})  and (\ref{equation55}).
\end{proof}
\begin{lemma}{\label{Lemma10}}\ Let $D_{4}$ be one
variant of the coefficients given by $D_{4}
:=\max_{K\in\mathcal{T}_{h}}c_{{\bf w},r,K}$. Under the assumption
of Lemma \ref{Lemma8}, it holds
\begin{equation}\label{equation56}
\begin{array}{lll}
\displaystyle\sum\limits_{K\in\mathcal{T}_{h}}c_{{\bf
w},r,K}||p-p_{h}||^{2}&\leq&\displaystyle\sum\limits_{K\in\mathcal{T}_{H}}c_{{\bf
w},r,K}||p-p_{H}||_{K}^{2}-1/2\sum\limits_{K\in\mathcal{T}_{H}}c_{{\bf
w},r,K}||p_{h}-p_{H}||_{K}^{2}\vspace{2mm}\\
&+&2D_{4}D_{1}(h_{0}e_{h}^{2}+h_{0}^{2}(||f-f_{h}||^{2}+||h\nabla_{h}(S^{-1}{\bf
u}_{h}\cdot{\bf w})||^{2})).
\end{array}
\end{equation}
\end{lemma}
\begin{proof}\ For each element $K\in\mathcal{T}_{h}$, it holds
the following identity
\begin{equation}\label{equation57}
\begin{array}{lll}
||p-p_{h}||_{K}^{2}&=&||p-p_{H}||_{K}^{2}-||p_{h}-p_{H}||_{K}^{2}-2(p-p_{h},p_{h}-p_{H})_{K}\vspace{2mm}\\
&=&||p-p_{H}||_{K}^{2}-||p_{h}-p_{H}||_{K}^{2}-2(Q_{h}p-p_{h},p_{h}-p_{H})_{K}.
\end{array}
\end{equation}
Notice that $c_{{\bf
w},r,K}$ does not change from $\mathcal{T}_{H}$ to $\mathcal{T}_{h}$.
Summing (\ref{equation57}) by multiplying $c_{{\bf w},r,K}$ over all
elements $K\in\mathcal{T}_{h}$, we have
\begin{equation}\label{equation58}
\begin{array}{lll}
\displaystyle\sum\limits_{K\in\mathcal{T}_{h}}c_{{\bf
w},r,K}||p-p_{h}||_{K}^{2}&\leq&\displaystyle\sum\limits_{K\in\mathcal{T}_{H}}c_{{\bf
w},r,K}||p-p_{H}||_{K}^{2}-1/2\sum\limits_{K\in\mathcal{T}_{H}}c_{{\bf
w},r,K}||p_{h}-p_{H}||_{K}^{2}\vspace{2mm}\\
&+&2D_{4}||Q_{h}p-p_{h}||^{2}.
\end{array}
\end{equation}
The conclusion (\ref{equation56}) follows from
(\ref{equation58}) and (\ref{equation51}).
\end{proof}

In what follows, we show the reduction of the error. To this end,
set $\gamma_{1},\gamma_{2}, \varepsilon_{0}$, and $\delta_{1}$ to be
any positive constants, which will be determined below. Introduce the following
quantity:
\begin{equation}
\label{quantity}
A_{h}^{2} :=\delta_{1}(1-\varepsilon_{0})^{-1}||\nabla\cdot({\bf
u}-{\bf
u}_{h})||^{2}+\gamma_{1}\eta_{h}^{2}+\gamma_{2}(||f-f_{h}||^{2}+||h\nabla_{h}(S^{-1}{\bf
u}_{h}\cdot{\bf w})||^{2}),
\end{equation}
where $f_{h}$ is the $L^{2}-$projection of $f$ onto
$P_{0}(\mathcal{T}_{h})$. We note that the definition of $A_{H}$ is similar to
$A_{h}$.

\begin{theorem}\label{theorem1+}
Let $({\bf u}_{h},p_{h})\in RT_{0}(\mathcal{T}_{h})\times
P_{0}(\mathcal{T}_{h})$ and $({\bf u}_{H},p_{H})\in
RT_{0}(\mathcal{T}_{H})\times P_{0}(\mathcal{T}_{H})$ be the
approximation solutions to the stress and
displacement variables with respect to $\mathcal{T}_{h}$ and
$\mathcal{T}_{H}$, respectively. Denote by $e_{h}$ and $e_{H}$ the
errors of the stress and displacement variables with respect to
$\mathcal{T}_{h}$ and $\mathcal{T}_{H}$, respectively. Let $h_{0}$
be the mesh-size of the quasi-uniform initial mesh
$\mathcal{T}_{0}$, and $q$ and $\alpha\in(0,1)$   two constants to be determined below. Then it holds
\begin{equation}\label{equation47}
e_{h}^{2}+(1-h_{0}q)A_{h}^{2}\leq\alpha^{2}(e_{H}^{2}+(1-h_{0}q)A_{H}^{2})
\end{equation}
when
$h_{0}\leq\frac{1-\alpha^{2}}{1+\alpha^{2}}\frac{1}{q}$.
\end{theorem}
\begin{proof}\ For convenience, denote
\begin{equation*}
D_{5} :=2D_{1}+4\delta_{1}^{2}D_{1}D_{3}+2\delta_{1}D_{1}D_{4}.
\end{equation*}
Recalling the definition, (\ref{estimator1}),  of $E_{H}$, a combination
of (\ref{equation48}), (\ref{equation53}) and (\ref{equation56}) indicates
\begin{equation}\label{equation59}
\begin{array}{lll}
& &e_{h}^{2}+\delta_{1}||\nabla\cdot({\bf u}-{\bf u}_{h})||^{2}\leq
e_{H}^{2}+\delta_{1}||\nabla\cdot({\bf u}-{\bf
u}_{H})||^{2}-1/2E_{H}^{2}+h_{0}D_{5}\delta_{1}^{-1}e_{h}^{2}\vspace{2mm}\\
& & \hspace{2.8cm}+4\delta_{1}D_{2}e_{h}^{2}
+D_{5}\delta_{1}^{-1}h_{0}^{2}(||f-f_{h}||^{2}+||h\nabla_{h}(S^{-1}{\bf
u}_{h}\cdot{\bf w})||^{2}).
\end{array}
\end{equation}
For a constant $\varepsilon>0$ which will be determined below, denote
$\varepsilon_{0} :=\frac{\varepsilon}{1+\varepsilon}$. We firstly
choose $\delta_{1}$ with $4\delta_{1}D_{2}\leq\varepsilon_{0}$.
The reliable estimate, (\ref{equation39}), of the stress and displacement variables, together with (\ref{equation59}),
implies
\begin{equation*}
\begin{array}{lll}
e_{h}^{2}+\delta_{1}||\nabla\cdot({\bf u}-{\bf u}_{h})||^{2}&\leq&
e_{H}^{2}+\delta_{1}||\nabla\cdot({\bf u}-{\bf
u}_{H})||^{2}-1/2E_{H}^{2}+C_{1}D_{5}\delta_{1}^{-1}h_{0}\eta_{h}^{2}\vspace{2mm}\\
& &+
\varepsilon_{0}e_{h}^{2}+D_{5}\delta_{1}^{-1}h_{0}^{2}(||f-f_{h}||^{2}+||h\nabla_{h}(S^{-1}{\bf
u}_{h}\cdot{\bf w})||^{2}),
\end{array}
\end{equation*}
which results in the following inequality:
\begin{equation}\label{equation60}
\begin{array}{lll}
& &e_{h}^{2}+\frac{\delta_{1}}{1-\varepsilon_{0}}||\nabla\cdot({\bf
u}-{\bf u}_{h})||^{2}\leq(1+\varepsilon)
e_{H}^{2}+\frac{\delta_{1}}{1-\varepsilon_{0}}||\nabla\cdot({\bf
u}-{\bf
u}_{H})||^{2}-\frac{1}{2(1-\varepsilon_{0})}E_{H}^{2}\vspace{2mm}\\
&
&\hspace{2.1cm}+\frac{C_{1}D_{5}}{\delta_{1}(1-\varepsilon_{0})}h_{0}\eta_{h}^{2}
+\frac{D_{5}}{\delta_{1}(1-\varepsilon_{0})}h_{0}^{2}(||f-f_{h}||^{2}+||h\nabla_{h}(S^{-1}{\bf
u}_{h}\cdot{\bf w})||^{2}).
\end{array}
\end{equation}

According to   triangle inequality and inverse inequality, it
holds, for each $K\in\mathcal{T}_{h}$,
\begin{equation}\label{equation61-}
h_{K}||\nabla(S^{-1}{\bf u}_{h}\cdot{\bf w})||_{K}\leq
h_{K}||\nabla(S^{-1}{\bf u}_{H}\cdot{\bf
w})||_{K}+C_{7}||S^{-1}({\bf u}_{h}-{\bf u}_{H})\cdot{\bf w}||_{K}.
\end{equation}
Notice that $||f-f_{h}||_{K}\leq||f-f_{H}||_{K}$ for all
$K\in\mathcal{T}_{h}$. For any given $\delta_{3}>0$ which will be
determined below, (\ref{equation61-}) and Young inequality imply
\begin{equation}\label{equation61-+}
\begin{array}{lll}
||f-f_{h}||^{2}+||h\nabla_{h}(S^{-1}{\bf u}_{h}\cdot{\bf
w})||^{2}&\leq&(1+\delta_{3})(||f-f_{H}||^{2}+||H\nabla_{H}(S^{-1}{\bf
u}_{H}\cdot{\bf w})||^{2})\vspace{2mm}\\
& &\ \ +(1+\delta_{3}^{-1})C_{7}^{2}D_{2}||S^{-1/2}({\bf u}_{h}-{\bf
u}_{H})||^{2}.
\end{array}
\end{equation}

From the definition, (\ref{quantity}), of $A_{h}^{2}$, the estimator reduction (\ref{equation40}) with the
marking strategy, the estimates (\ref{equation60})  and (\ref{equation61-+}),
and the fact $||S^{-1/2}({\bf u}_{h}-{\bf
u}_{H})||^{2}\leq E_{H}^{2}$, it holds, for any given $\delta_{2}>0$ which
will be determined below,
\begin{equation}\label{equation61}
\begin{array}{lll}
& &e_{h}^{2}+A_{h}^{2}\leq(1+\varepsilon)
e_{H}^{2}+\frac{\delta_{1}}{1-\varepsilon_{0}}||\nabla\cdot({\bf
u}-{\bf u}_{H})||^{2}-\frac{1}{2(1-\varepsilon_{0})}E_{H}^{2}+
\frac{C_{1}D_{5}}{\delta_{1}(1-\varepsilon_{0})}h_{0}\eta_{h}^{2}\vspace{2mm}\\
&
&+C_{3}^{2}(1+\delta_{2}^{-1})D_{\mathcal{T}_{0}}^{2}\gamma_{1}E_{H}^{2}
+\gamma_{1}(1+\delta_{2})(1-\lambda\theta^{2})\eta_{H}^{2}+
\gamma_{2}C_{7}^{2}D_{2}(1+\frac{1}{\delta_{3}})E_{H}^{2}\vspace{2mm}\\
&
&+\frac{D_{5}}{\delta_{1}(1-\varepsilon_{0})}h_{0}^{2}(||f-f_{h}||^{2}+||h\nabla_{h}(S^{-1}{\bf
u}_{h}\cdot{\bf w})||^{2})\vspace{2mm}\\
&
&+\gamma_{2}(1+\delta_{3})(||f-f_{H}||^{2}+||H\nabla_{H}(S^{-1}{\bf
u}_{H}\cdot{\bf w})||^{2}).
\end{array}
\end{equation}
We next choose $\gamma_{1}$ and $\gamma_{2}$ such that
\begin{equation*}
\gamma_{1}C_{3}^{2}(1+\delta_{2}^{-1})D_{\mathcal{T}_{0}}^{2}=\frac{1}{4(1-\varepsilon_{0})},
\quad
\gamma_{2}C_{7}^{2}D_{2}(1+\delta_{3}^{-1})=\frac{1}{4(1-\varepsilon_{0})}.
\end{equation*}
Then it follows
\begin{equation}\label{equation62}
\begin{array}{lll}
e_{h}^{2}+A_{h}^{2}&\leq&(1+\varepsilon)
e_{H}^{2}+\frac{\delta_{1}}{1-\varepsilon_{0}}||\nabla\cdot({\bf
u}-{\bf u}_{H})||^{2}+
\gamma_{1}(1+\delta_{2})(1-\lambda\theta^{2})\eta_{H}^{2}\vspace{2mm}\\
& &
+\frac{D_{5}}{\delta_{1}(1-\varepsilon_{0})}h_{0}^{2}(||f-f_{h}||^{2}+||h\nabla_{h}(S^{-1}{\bf
u}_{h}\cdot{\bf w})||^{2})
+\frac{C_{1}D_{5}}{\delta_{1}(1-\varepsilon_{0})}h_{0}\eta_{h}^{2}\vspace{2mm}\\
&
&+\gamma_{2}(1+\delta_{3})(||f-f_{H}||^{2}+||H\nabla_{H}(S^{-1}{\bf
u}_{H}\cdot{\bf w})||^{2}).
\end{array}
\end{equation}

For any given $\delta_{4},\delta_{5}>0$ which will be determined below,
the reliable estimate (\ref{equation39}) on $\mathcal{T}_{H}$, i.e.  $e_{H}^{2}\leq C_{1}\eta_{H}^{2}$, and
the above estimate (\ref{equation62}), indicate
\begin{equation}\label{equation63}
\begin{array}{lll}
&
&e_{h}^{2}+A_{h}^{2}\leq(1+\varepsilon-1/2\lambda\theta^{2}\gamma_{1}(1+\delta_{2})C_{1}^{-1})
e_{H}^{2}+\gamma_{1}(1+\delta_{2})(1-1/2\lambda\theta^{2})\eta_{H}^{2}\vspace{2mm}\\
& &\
+\frac{\delta_{1}(1-\delta_{4})}{1-\varepsilon_{0}}||\nabla\cdot({\bf
u}-{\bf
u}_{H})||^{2}+\frac{D_{5}}{\delta_{1}(1-\varepsilon_{0})}h_{0}^{2}(||f-f_{h}||^{2}+||h\nabla_{h}(S^{-1}{\bf
u}_{h}\cdot{\bf
w})||^{2})\vspace{2mm}\\
& &\
+\frac{C_{1}D_{5}}{\delta_{1}(1-\varepsilon_{0})}h_{0}\eta_{h}^{2}
+(1-\delta_{5})\gamma_{2}(1+\delta_{3})(||f-f_{H}||^{2}+||H\nabla_{H}(S^{-1}{\bf
u}_{H}\cdot{\bf w})||^{2})\vspace{2mm}\\
& &\
+\frac{\delta_{1}\delta_{4}}{1-\varepsilon_{0}}||\nabla\cdot({\bf
u}-{\bf
u}_{H})||^{2}+(1+\delta_{3})\delta_{5}\gamma_{2}(||f-f_{H}||^{2}+||H\nabla_{H}(S^{-1}{\bf
u}_{H}\cdot{\bf w})||^{2}).
\end{array}
\end{equation}
Now we fix a sufficiently small $\delta_{2}$ and, subsequently, a
sufficiently small $\varepsilon$ such that
\begin{equation*}
\tilde{\alpha}^{2}
:=\max(1+\varepsilon-\frac{C_{1}}{2}\lambda\theta^{2}\gamma_{1}(1+\delta_{2}),
(1+\delta_{2})(1-\frac{\lambda\theta^{2}}{2}),1-\delta_{4},(1-\delta_{5})(1+\delta_{3}))<1.
\end{equation*}

Let $D_{6}$ be one variant of the coefficients given by
\begin{equation*}
D_{6} :=4\max_{K\in\mathcal{T}_{H}}(||{\bf
w}||_{L^{\infty}(K)}^{2}c_{S,K}^{-1},C_{{\bf w},r,K}^{2}c_{{\bf
w},r,K}^{-1},1).
\end{equation*}
From (\ref{equation20}), we get
\begin{equation}\label{equation63+}
||\nabla\cdot({\bf u}-{\bf u}_{H})||^{2}\leq
D_{6}(||f-f_{H}||^{2}+||H\nabla_{H}(S^{-1}{\bf u}_{H}\cdot{\bf
w})||^{2}+\eta_{H}^{2}).
\end{equation}

We further choose $\delta_{i}(i=3,4,5)$ such that
\begin{equation*}
\frac{\delta_{1}\delta_{4}}{1-\varepsilon_{0}}D_{6}\leq\gamma_{2}h_{0},\
\quad \ (1+\delta_{3})\delta_{5}\leq C_{8}h_{0}.
\end{equation*}
In fact, we may firstly fix $\delta_{3}$ satisfying
$\delta_{3}<C_{8}h_{0}$, then   choose $\delta_{5}$ such that
\begin{equation*}
\frac{\delta_{3}}{1+\delta_{3}}<\delta_{5}<\min(1,\frac{C_{8}h_{0}}{1+\delta_{3}}).
\end{equation*}
Finally, by noticing the choice of $\gamma_{2}$, we can choose
$\delta_{4}$ with
\begin{equation*}
\delta_{4}<\min(1,\frac{h_{0}}{D_{6}}\frac{1-\varepsilon_{0}}{\delta_{1}}
\frac{1}{4C_{7}^{2}D_{2}(1-\varepsilon_{0})(1+\delta_{3}^{-1})}).
\end{equation*}
These
choices,  together with (\ref{equation63}) and (\ref{equation63+}),  lead to
\begin{equation}\label{equation64}
\begin{array}{lll}
e_{h}^{2}+A_{h}^{2}&\leq&\tilde{\alpha}^{2}(e_{H}^{2}+A_{H}^{2})+
\frac{C_{1}D_{5}}{\delta_{1}(1-\varepsilon_{0})}h_{0}\eta_{h}^{2}\vspace{2mm}\\
&
&+\frac{D_{5}}{\delta_{1}(1-\varepsilon_{0})}h_{0}^{2}(||f-f_{h}||^{2}+
||h\nabla_{h}(S^{-1}{\bf u}_{h}\cdot{\bf w})||^{2})\vspace{2mm}\\
&
&+\gamma_{2}h_{0}(1+C_{8})(||f-f_{H}||^{2}+||H\nabla_{H}(S^{-1}{\bf
u}_{H}\cdot{\bf w})||^{2})+\gamma_{2}h_{0}\eta_{H}^{2}.
\end{array}
\end{equation}

Let $q$ be one variant of the coefficients given by
\begin{equation*}
q :=\max(C_{8}+1,\frac{\gamma_{2}}{\gamma_{1}},
\frac{C_{1}D_{5}}{\delta_{1}(1-\varepsilon_{0})}\frac{1}{\gamma_{1}},
\frac{D_{5}}{\delta_{1}(1-\varepsilon_{0})}\frac{h_{0}}{\gamma_{2}}).
\end{equation*}
From (\ref{equation64}) we arrive at
\begin{equation*}
e_{h}^{2}+A_{h}^{2}\leq\tilde{\alpha}^{2}(e_{H}^{2}+A_{H}^{2})+qh_{0}(A_{h}^{2}+A_{H}^{2}),
\end{equation*}
which implies
\begin{equation}\label{equation65}
e_{h}^{2}+(1-qh_{0})A_{h}^{2}\leq\tilde{\alpha}^{2}e_{H}^{2}+
(\tilde{\alpha}^{2}+qh_{0})A_{H}^{2}.
\end{equation}
We finally choose $h_{0}$ such that
\begin{equation*}
0<\frac{\tilde{\alpha}^{2}+qh_{0}}{1-qh_{0}}\leq\alpha^{2}
:=\frac{1+\tilde{\alpha}^{2}}{2},
\end{equation*}
which yields the assertion (\ref{equation47}) with
$h_{0}\leq(1-\alpha^{2})/(q(1+\alpha^{2}))$.
\end{proof}
\begin{remark}\ {\rm(Choices of the initial mesh size)}\  Some
simple calculations show
$$
q\leq\max\{D(\delta_{1},\delta_{2}),
\frac{D_{5}h_{0}C_{7}^{-2}}{4\delta_{1}(1-\varepsilon_{0})^{2}D_{2}}\}
$$
with
$$D(\delta_{1},\delta_{2})
:=\max\{C_{8}+1,\frac{C_{3}^{2}(1+\delta_{2}^{-1})D_{\mathcal{T}_{0}}^{2}}{C_{7}^{2}D_{2}},
\frac{C_{2}D_{5}}{4\delta_{1}(1-\varepsilon_{0})C_{3}^{2}(1+
\delta_{2}^{-1})D_{\mathcal{T}_{0}}^{2}}\}.$$
Then it holds
$$\frac{1}{q}\geq \min\{\frac{1}{D(\delta_{1},\delta_{2})},\frac{4\delta_{1}(1-\varepsilon_{0})^{2}C_{7}^{2}D_{2}}{D_{5}h_{0}}\},$$
which indicates
$$\frac{1}{q}\geq  \frac{1}{D(\delta_{1},\delta_{2})}\quad \text{if} \quad h_0\leq \frac{4\delta_{1}(1-\varepsilon_{0})^{2}C_{7}^{2}D_{2}}{D_5}.$$

As required in Theorem \ref{theorem1+},   the initial mesh size $h_0$ is assumed to satisfy
$h_{0}\leq \frac{1-\alpha^{2}}{q(1+\alpha^{2})}$.
Then eventually we may choose $h_0$ with
$$h_{0}\leq \min\{\frac{1-\alpha^{2}}{(1+\alpha^{2})D(\delta_{1},\delta_{2})},\frac{4\delta_{1}(1-\varepsilon_{0})^{2}C_{7}^{2}D_{2}}{D_5}\}.$$

\end{remark}

{\bf The proof of Theorem \ref{theorem--}}\ Theorem \ref{theorem1+}
shows that the error of the stress and displacement variables plus
the quantity $A_h^2$ uniformly reduces with a fixed factor $\alpha^{2}$
between two successive meshes. Replace the subscripts
$H$ and $h$ respectively  by the iteration counters   $k$ and $k+1$, we then obtain Theorem
\ref{theorem--}   directly from  Theorem \ref{theorem1+}.

\section{Numerical experiments}
In this section, we test the performance of the adaptive algorithm
AMFEM described in section 2 with four model problems. We are thus able to study how meshes
adapt to various effects from lack of regularity of solutions and
convexity of domains to data smoothness, boundary layers and changing
boundary conditions. We note that the implementation of AMFEM is done  without enforcing the interior node property
in the refinement step.

\subsection{Model problem with singularity at the origin}
We consider the problem (\ref{equation1}) in an $L$-shape domain
$\Omega=\{(-1,1)\times(0,1)\}\cup\{(-1,0)\times(-1,0)\}$ with ${\bf
w}= r=0$ and $f=0$. The exact solution  is given by
\begin{equation*}
p(\rho,\theta)=\rho^{2/3}\sin(2\theta/3),
\end{equation*}
where $\rho,\theta$ are the polar coordinates.
\begin{figure}[htbp]
  \begin{minipage}[t]{0.5\linewidth}
    \centering
    \includegraphics[width=2.5in]{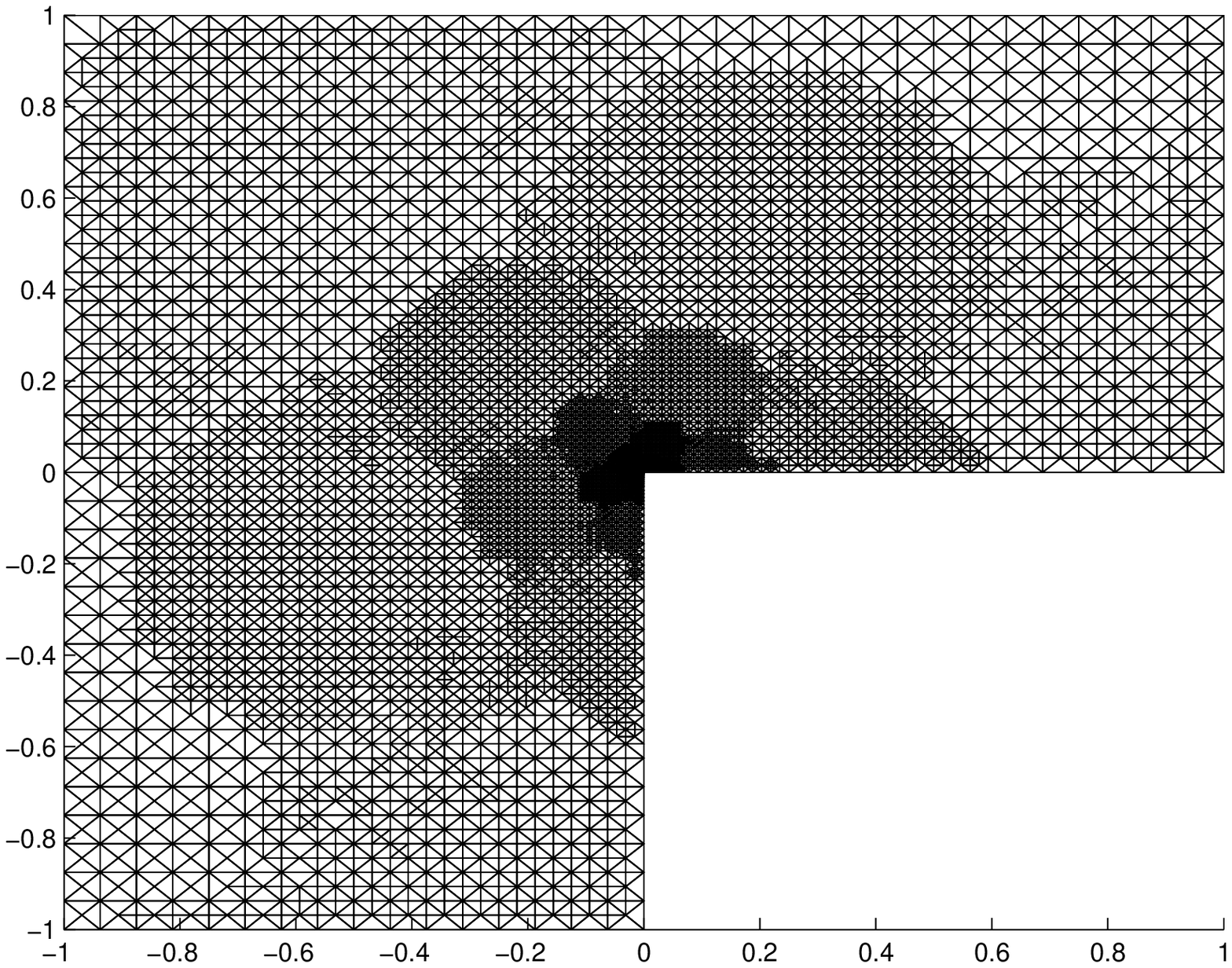}\\
  \end{minipage}
  \begin{minipage}[t]{0.5\linewidth}
    \centering
    \includegraphics[width=2.5in]{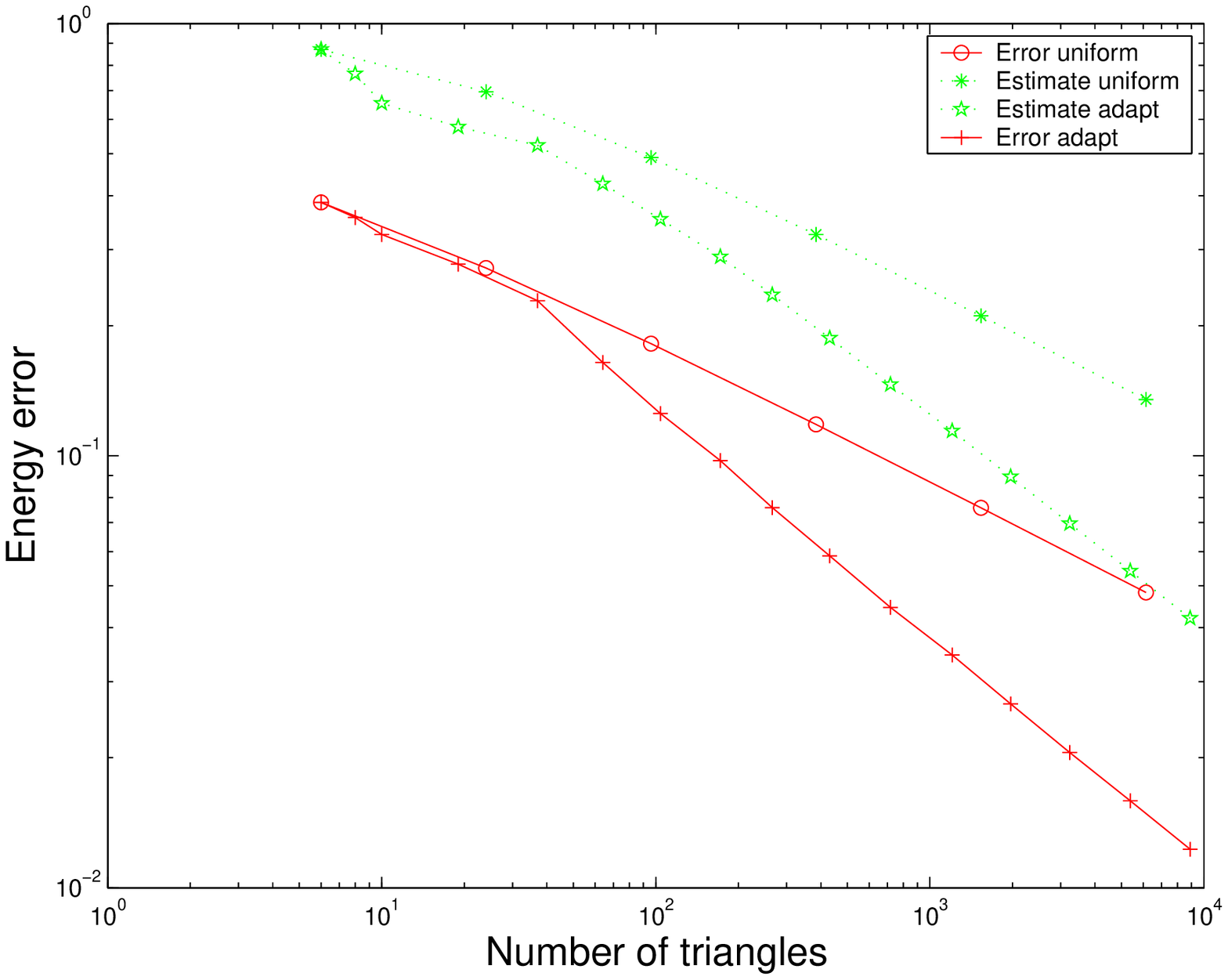}\\
  \end{minipage}
\addtocontents{lof}{figure}{FIG 6.1. {\small {\it A mesh with 14692
triangles
 (left) and the estimated and actual errors
  in uniformly / adaptively refined meshes (right) for the marking parameter $\theta=0.5$.}}}\\
\end{figure}

Since this model possesses singularity at the origin, we see in the
left figure of Fig 6.1 that the refinement concentrates around the
origin, which means the predicted error estimator captures well the
singularity of the solution. The right graph of Fig 6.1 reports the
estimated and actual errors of the numerical solutions on uniformly
and adaptively refined meshes. It can be seen that the error of the
stress and displacement in $L^{2}$ norm uniformly reduces with a
fixed factor on two successive meshes after several steps of
iterations, and that the error on the adaptively refined meshes
decreases more rapidly than the one on the uniformly refined meshes.
This shows that the adaptive mixed finite element method is
convergent with respect to the energy error.

\subsection{Model problem with inhomogeneous diffusion tensor}
We consider the problem (\ref{equation1}) in a square domain
$\Omega=(-1,1)\times(-1,1)$ with ${\bf w}=r=0$ and $f=0$, where
$\Omega$ is divided into four subdomains $\Omega_{i}$ ($i=1,2,3,4$)
corresponding to the axis quadrants (in the counterclockwise
direction), and the diffusion-dispersion tensor $S$ is piecewise
constant with $S=s_{i}I$ in $\Omega_{i}$. This model problem is
taken from \cite{Eigestad,Riviere;Wheeler2003,Vohralik1}. We suppose
the exact solution of this model has the form
\begin{equation*}
p(\rho,\theta)=\rho^{\alpha}(a_{i}sin(\alpha\theta)+b_{i}cos(\alpha\theta))
\end{equation*}
in each $\Omega_{i}$ with Dirichlet boundary conditions. Here
$\rho,\theta$ are the polar coordinates in $\Omega$, $a_{i}$ and
$b_{i}$ are constants depending on $\Omega_{i}$, and $\alpha$ is a
parameter. We note that The stress solution, ${\bf
u}=-S\nabla p$, is not continuous across the interfaces,
and only its normal component    is continuous. It finally exhibits a strong
singularity at the origin. We consider two sets of coefficients in
the following table:
\begin{center}
\scriptsize
\begin{tabular}{|c|c|} \hline
Case 1 &Case 2\\ \hline $s_{1}=s_{3}=5$,
$s_{2}=s_{4}=1$&$s_{1}=s_{3}=100$, $s_{2}=s_{4}=1$\\ \hline
$\alpha=0.53544095$ &$\alpha=0.12690207$\\ \hline
$a_{1}=\ \ 0.44721360$, $b_{1}=\ \ 1.00000000$&$a_{1}=\ \ 0.10000000$, $b_{1}=\ \ 1.00000000$\\
$a_{2}=-0.74535599$,  $b_{2}=\ \ 2.33333333$&$a_{2}=-9.60396040$, $b_{2}=\ \ 2.96039604$\\
$a_{3}=-0.94411759$,  $b_{3}=\ \ 0.55555555$&$a_{3}=-0.48035487$,  $b_{3}=-0.88275659$\\
$a_{4}=-2.40170264$, $b_{4}=-0.48148148$&$a_{4}=\ \ 7.70156488$,
$b_{4}=-6.45646175$\\ \hline
\end{tabular}
\end{center}

In MARK step, the marking parameter $\theta$,  in terms of D\"{o}rfler
marking, is chosen as $0.7$ in
the first case and as $0.94$ in the second case.  Table 6.1 shows for Case 1 some results of the actual
error $e_{k}$, the a posteriori indicator $\eta_{k}$, the
experimental convergence rate, ${\rm EOC}_{E}$, of $E_{k}$, and the
experimental convergence rate, ${\rm EOC}_{\eta}$, of $\eta_k$,
where
\begin{equation*}
{\rm EOC}_{E} :=\frac{\log(e_{k-1}/e_{k})}{\log({\rm DOF}_{k}/{\rm
DOF}_{k-1})},\ \ \ {\rm EOC}_{\eta}
:=\frac{\log(\eta_{k-1}/\eta_{k})}{\log({\rm DOF}_{k}/{\rm
DOF}_{k-1})},
\end{equation*}
  and ${\rm DOF}_{k}$ denotes the number of elements with
respect to the $k-$th iteration. We can see that the convergence
rates ${\rm EOC}_{E}$ and ${\rm EOC}_{\eta}$ are close to 0.5 as the
iteration number $k=15$, which means the optimal decays of the actual
error $e_k$ and the a posteriori error indicator $\eta_{k}$ are almost attained
after 15 iterations with optimal meshes.
\begin{table}[!h]\renewcommand{\baselinestretch}{1.25}\small
\begin{center}
 \caption{Results of actual error $E_{k}$, a
 posteriori indicator $\eta_{k}$, and their convergence rates ${\rm
EOC}_{E}$ and ${\rm EOC}_{\eta}$: Case 1}
\small 
\begin{tabular}{|c|c|c|c|c|c|} \hline
$k$& ${\rm DOF}_{k}$& $e_{k}$& $\eta_{k}$ &${\rm EOC}_{E}$& ${\rm
EOC}_{\eta}$\\ \hline 1&8&1.3665&5.0938&$-$&$-$\\ \hline
2&20&1.1346&3.4700&0.2030&0.4189\\ \hline
9&2235&0.1776&1.1115&0.4016&0.4004\\ \hline
11&7165&0.1106&0.7111&0.3851&0.4004\\ \hline
12&13188&0.0871&0.5566&0.3915&0.4015\\ \hline
14&43785&0.0510&0.3365&0.4707&0.4476\\ \hline
15&76770&0.0387&0.2581&0.4915&0.4724\\ \hline
\end{tabular}
\end{center}
\end{table}

Fig 6.2 shows an adaptively refined mesh with 4763 elements and the
estimated and actual errors against the number of elements in
adaptively refined meshes for Case 1. Fig 6.3 shows an adaptively
refined mesh with 1093 elements and the actual error against the
number of elements in adaptively refined meshes for Case 2.
\begin{figure}[htbp]
  \begin{minipage}[t]{0.5\linewidth}
    \centering
    \includegraphics[width=2.25in]{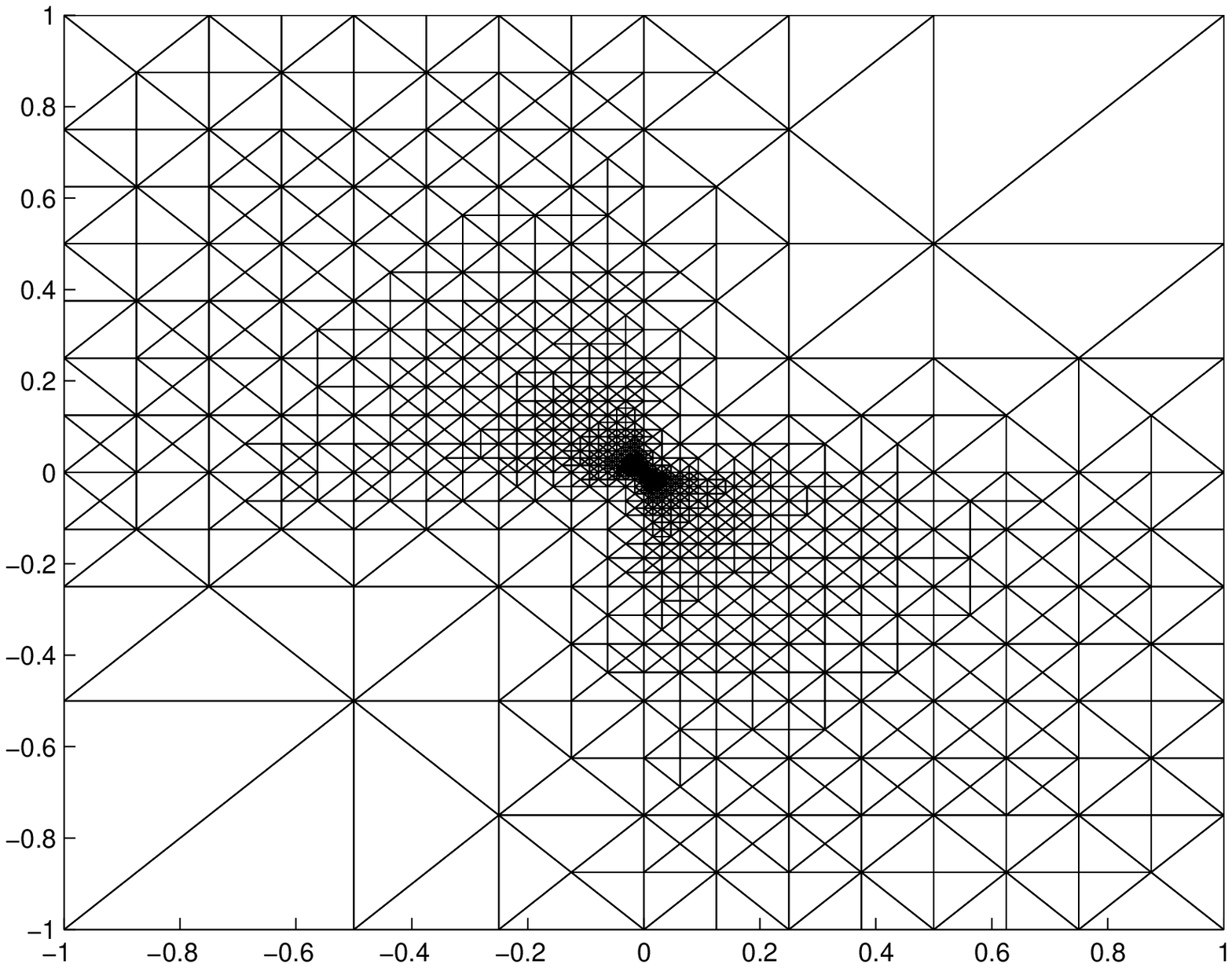}\\
  \end{minipage}
  \begin{minipage}[t]{0.5\linewidth}
    \centering
    \includegraphics[width=2.5in]{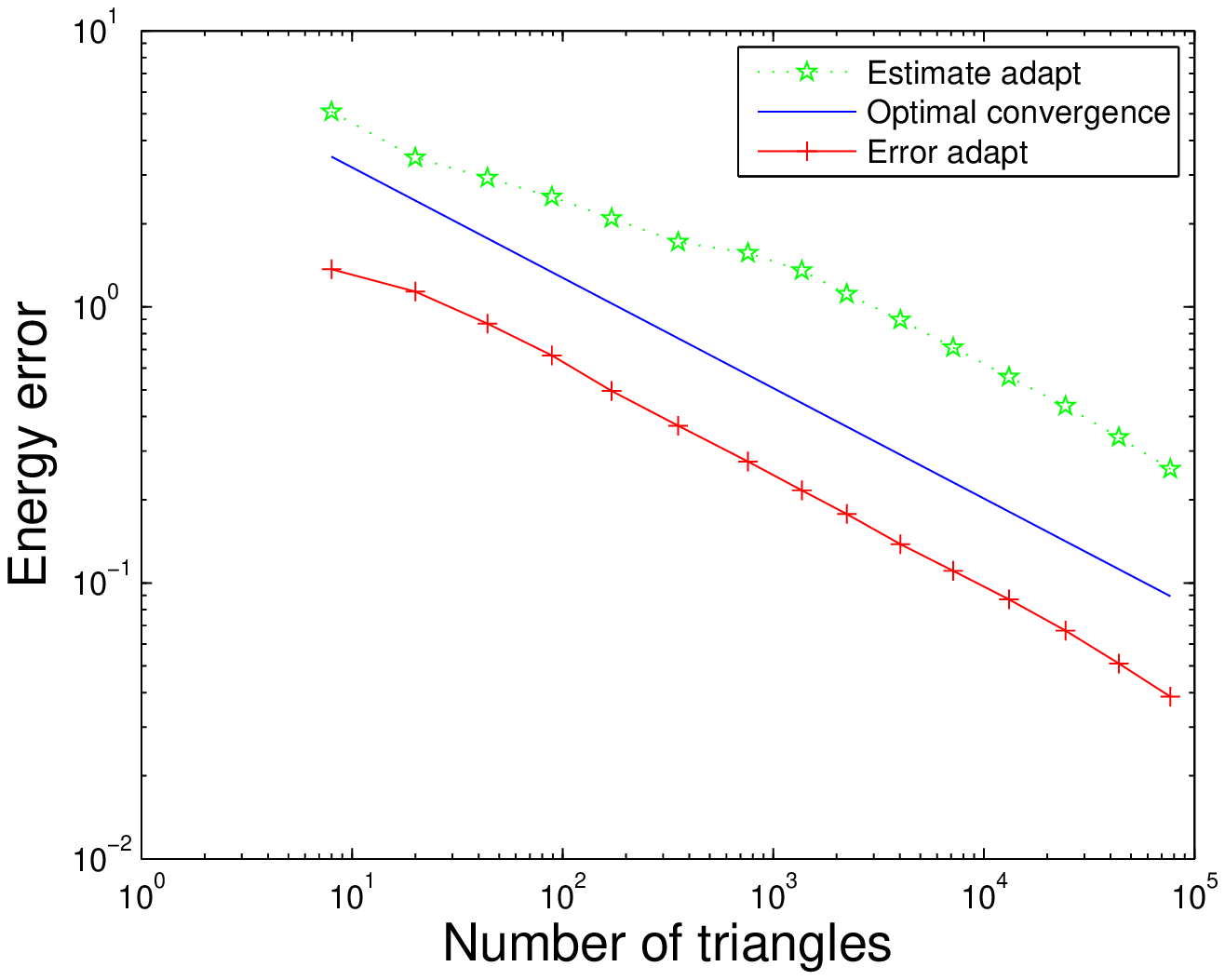}\\
  \end{minipage}
\addtocontents{lof}{figure}{FIG 6.2. {\small {\it A mesh with 4763
triangles
 (left) and the estimated and actual error
  against the number of elements in  adaptively refined meshes (right): Case 1.}}}\\
\end{figure}

From the left figures of Fig 6.2-6.3, we can see that the
refinement concentrates around the origin, which means the AMFEM
algorithm detects the region of rapid variation. In the right
graphs of Fig 6.2-6.3 each includes an optimal convergence line,
which shows in both cases, the energy error performs a trend of
descending with an optimal order convergent rate after several steps of
adaptive iterations for the problem with strongly discontinuous coefficients. We
note that the energy error is approximated with a 7-point quadrature
formula in each triangle.
\begin{figure}[htbp]
  \begin{minipage}[t]{0.5\linewidth}
    \centering
    \includegraphics[width=2.25in]{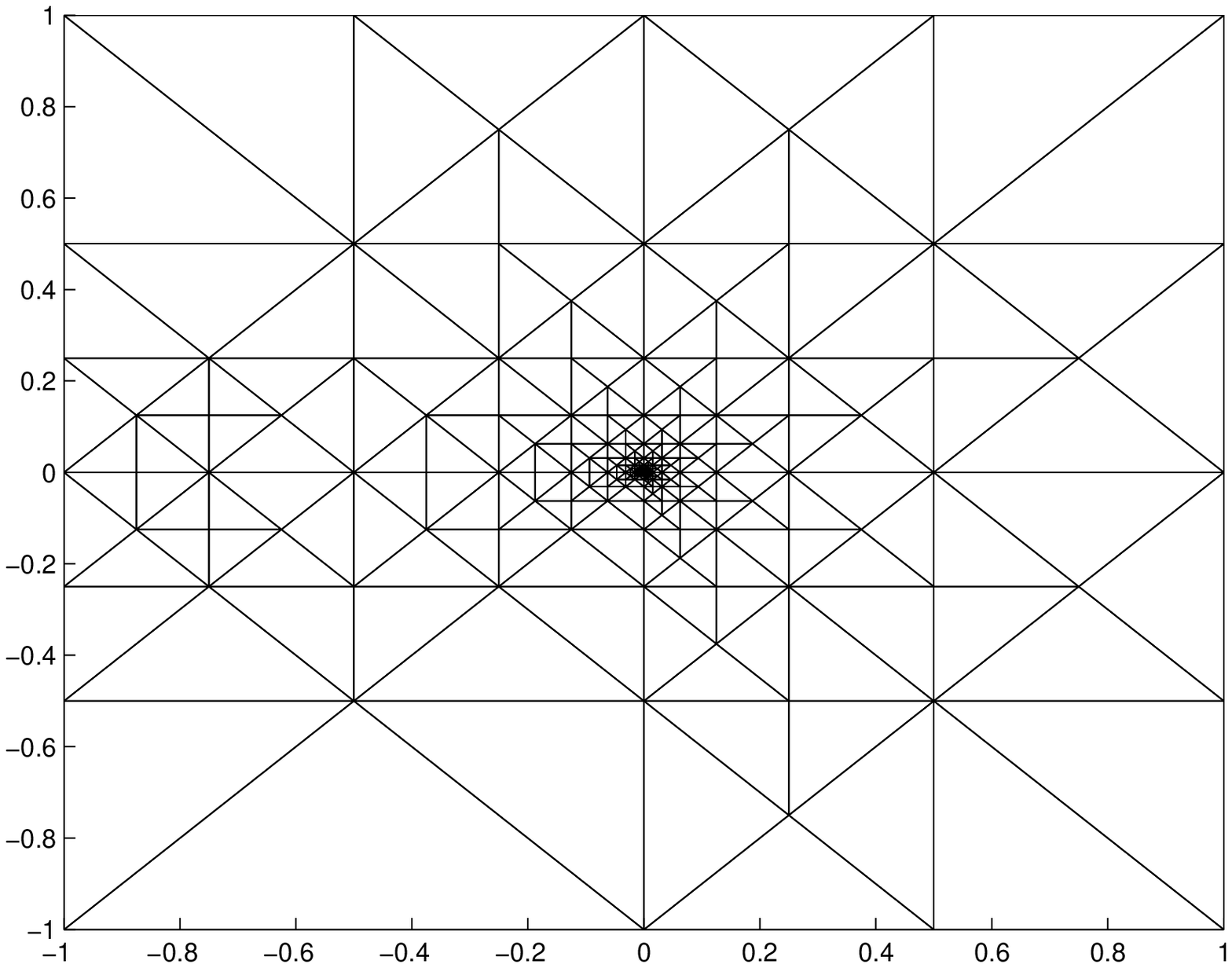}\\
  \end{minipage}
  \begin{minipage}[t]{0.5\linewidth}
    \centering
    \includegraphics[width=2.5in]{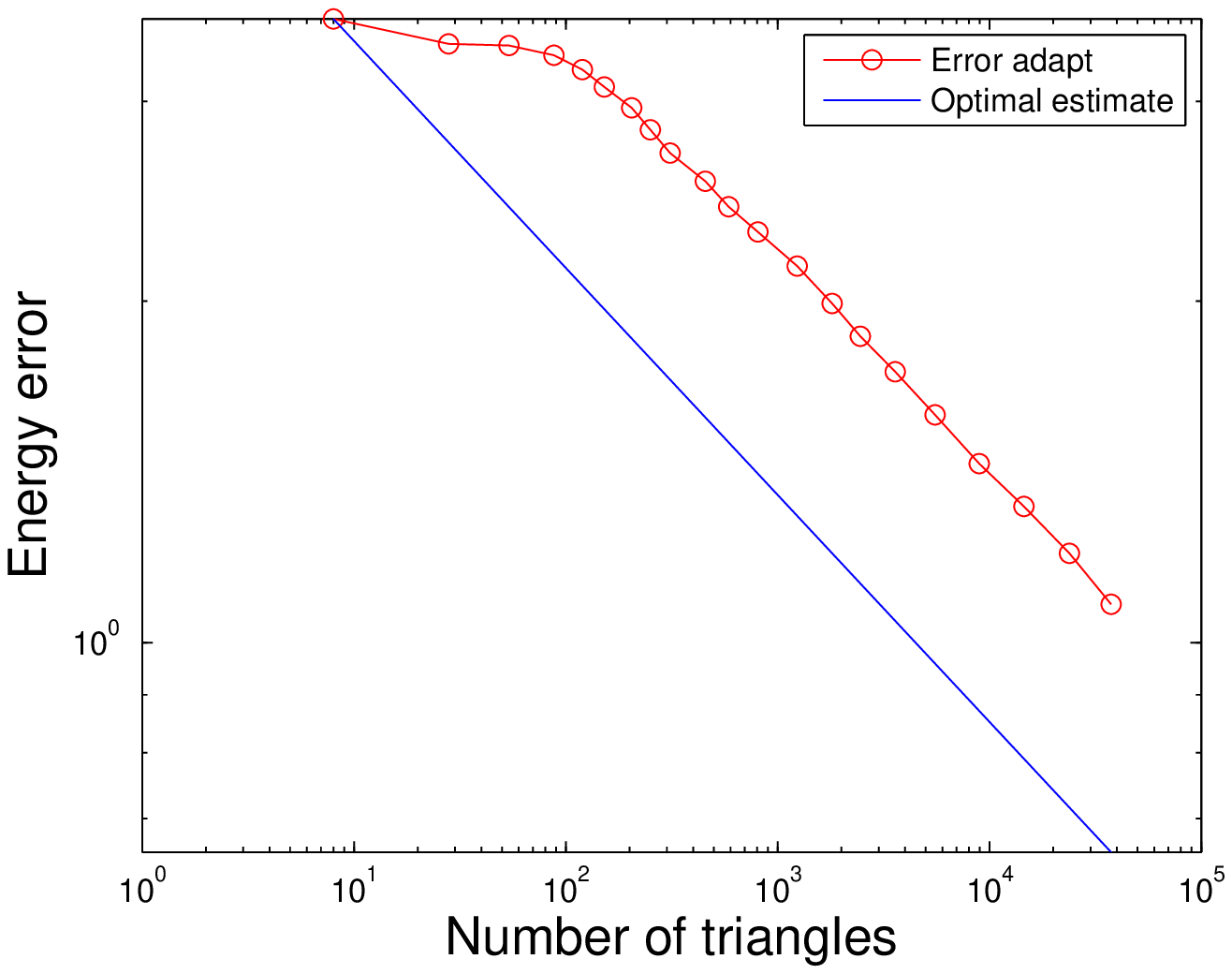}\\
  \end{minipage}
\addtocontents{lof}{figure}{FIG 6.3.\ {\small {\it A mesh with 1093
triangles
 (left) and the actual error
  against the number of elements in adaptively refined mesh (right): Case 2.}}}\\
\end{figure}

\subsection{Convection-dominated model problem: boundary layer}\ In this example, we take
$\Omega=(0,1)\times(0,1)$ in $\mathbb{R}^{2}$, and choose ${\bf
w}=(1,1)$ and $r=0$. Further, we set $p=0$ on $\partial\Omega$, and
select the right-hand side $f$ such that the analytical solution to
(\ref{equation1}) is given by
\begin{equation*}
p(x,y)=(\frac{\exp(\frac{x-1}{\varepsilon})-1}{\exp(-\frac{1}{\varepsilon})-1}+x-1)
(\frac{\exp(\frac{y-1}{\varepsilon})-1}{\exp(-\frac{1}{\varepsilon})-1}+y-1).
\end{equation*}
The solution is smooth, but has boundary layers at $x=1$ and $y=1$,
with layer width  of order $\mathcal{O}(\varepsilon)$. This
problem is well-suited to test whether the estimator is able to pick
up the steep gradients near these boundaries.

\begin{figure}[htbp]
  \begin{minipage}[t]{0.5\linewidth}
    \centering
    \includegraphics[width=2.25in]{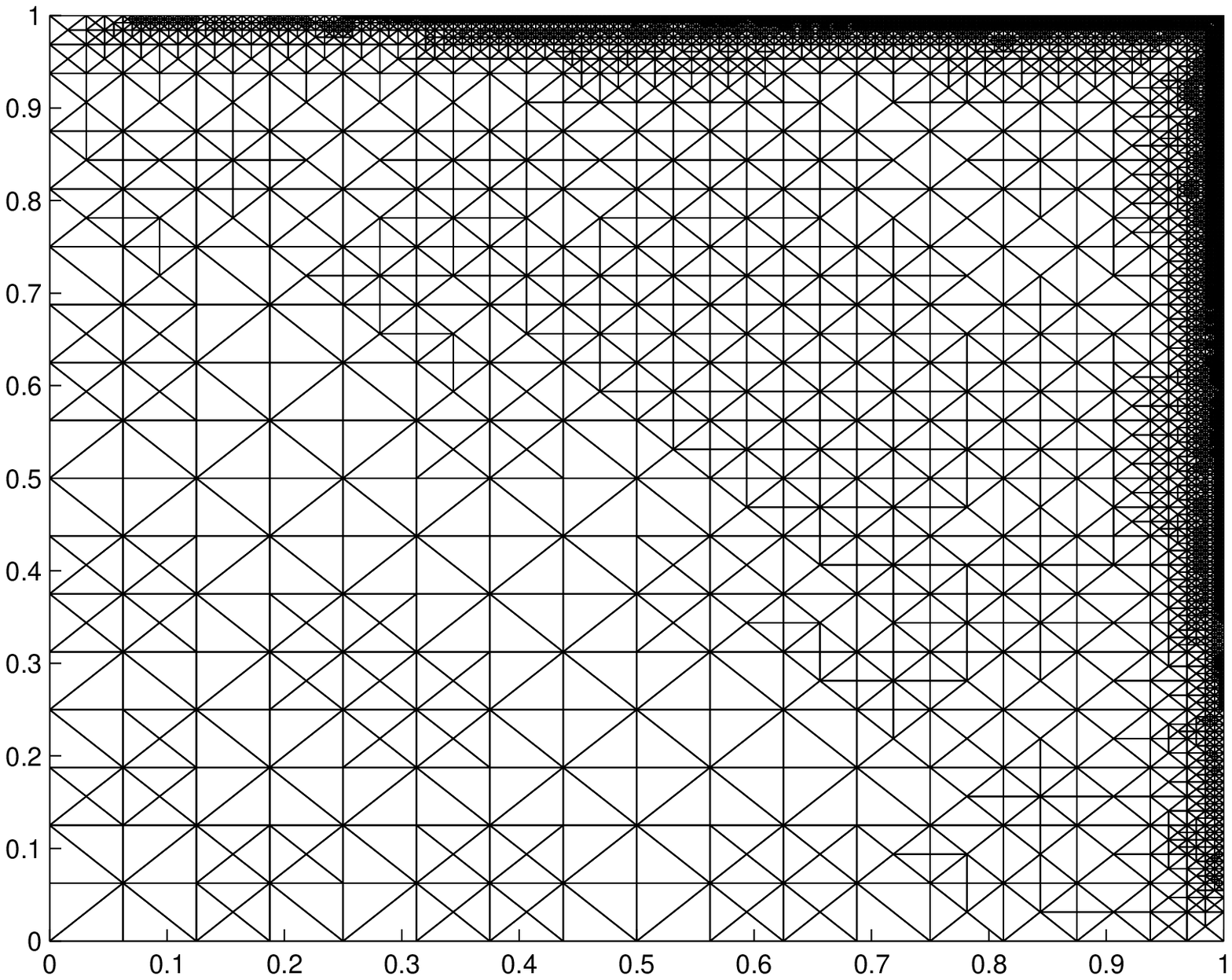}\\
  \end{minipage}
  \begin{minipage}[t]{0.5\linewidth}
    \centering
    \includegraphics[width=2.5in]{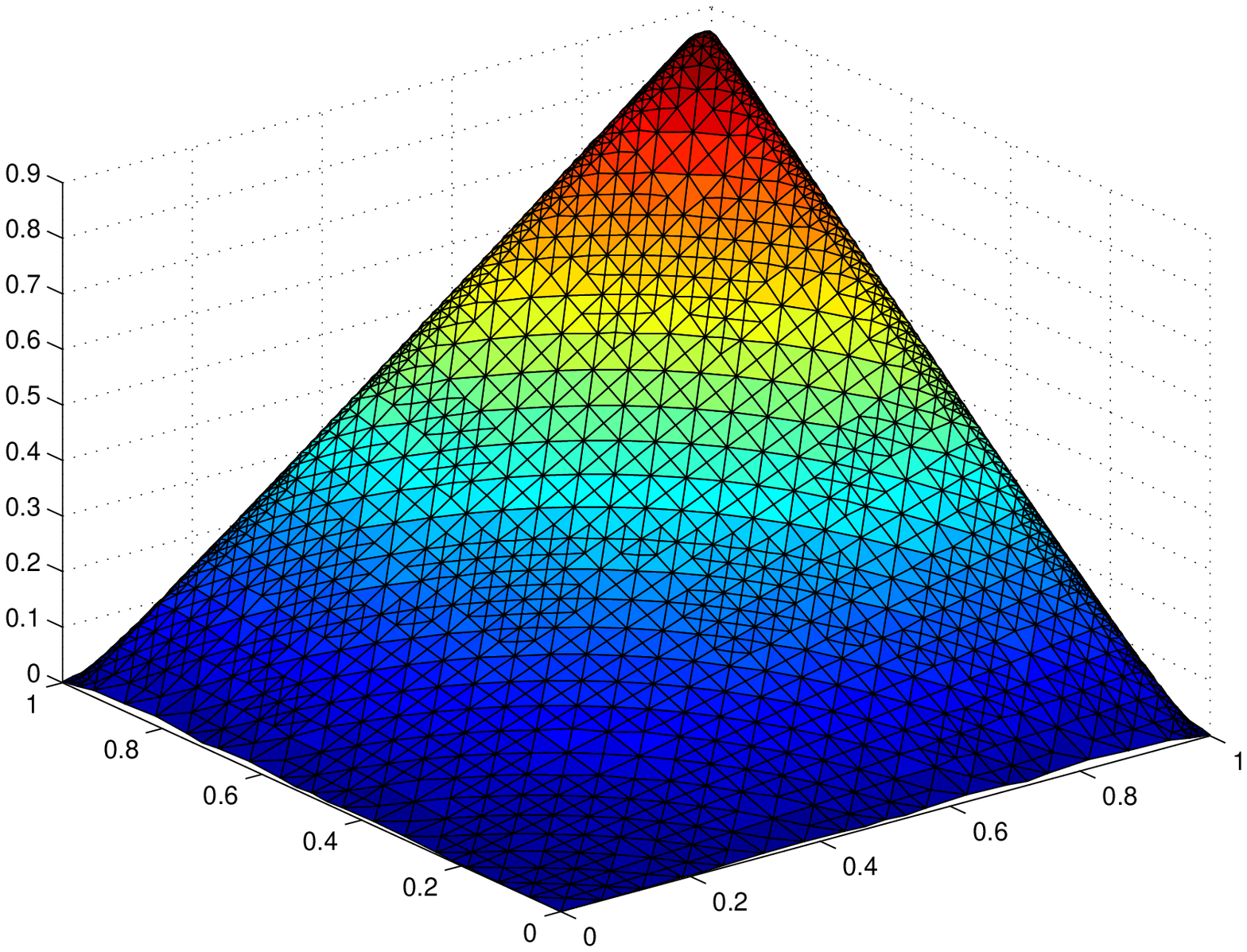}\\
  \end{minipage}
\addtocontents{lof}{figure}{FIG 6.4. {\small {\it A mesh with 10838
triangles
 (left) and postprocessing approximate displacement on the corresponding
 adaptively refined mesh (right) for $\varepsilon=0.01$.}}}\\
\end{figure}

We start computations from the origin mesh consisted of 8
right-angled triangles, and we choose the marking parameter $\theta=0.5$ in   the adaptive algorithm AMFEM.
\begin{figure}[htbp]
  \begin{minipage}[t]{0.5\linewidth}
    \centering
    \includegraphics[width=2.5in]{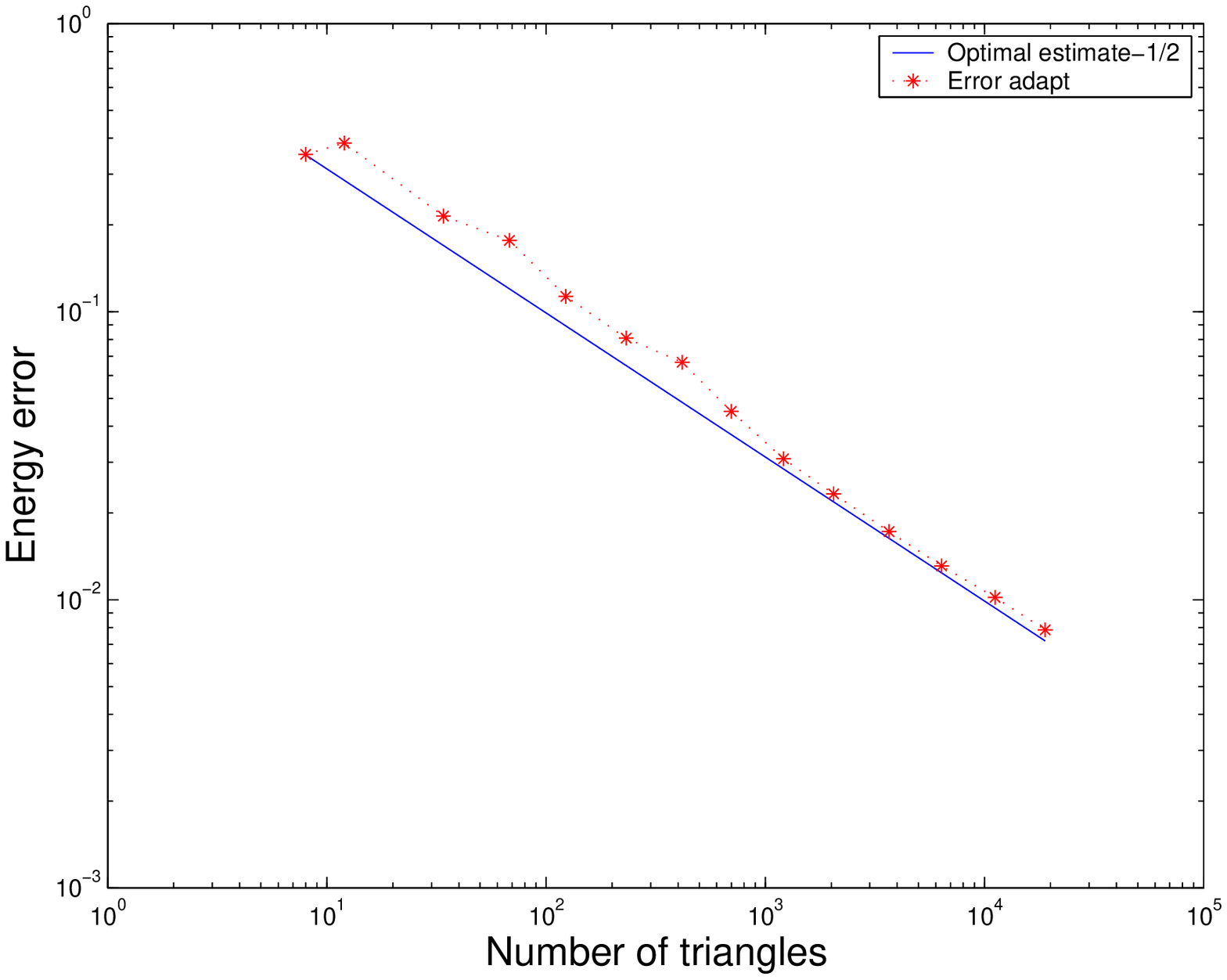}\\
  \end{minipage}
  \begin{minipage}[t]{0.5\linewidth}
    \centering
    \includegraphics[width=2.5in]{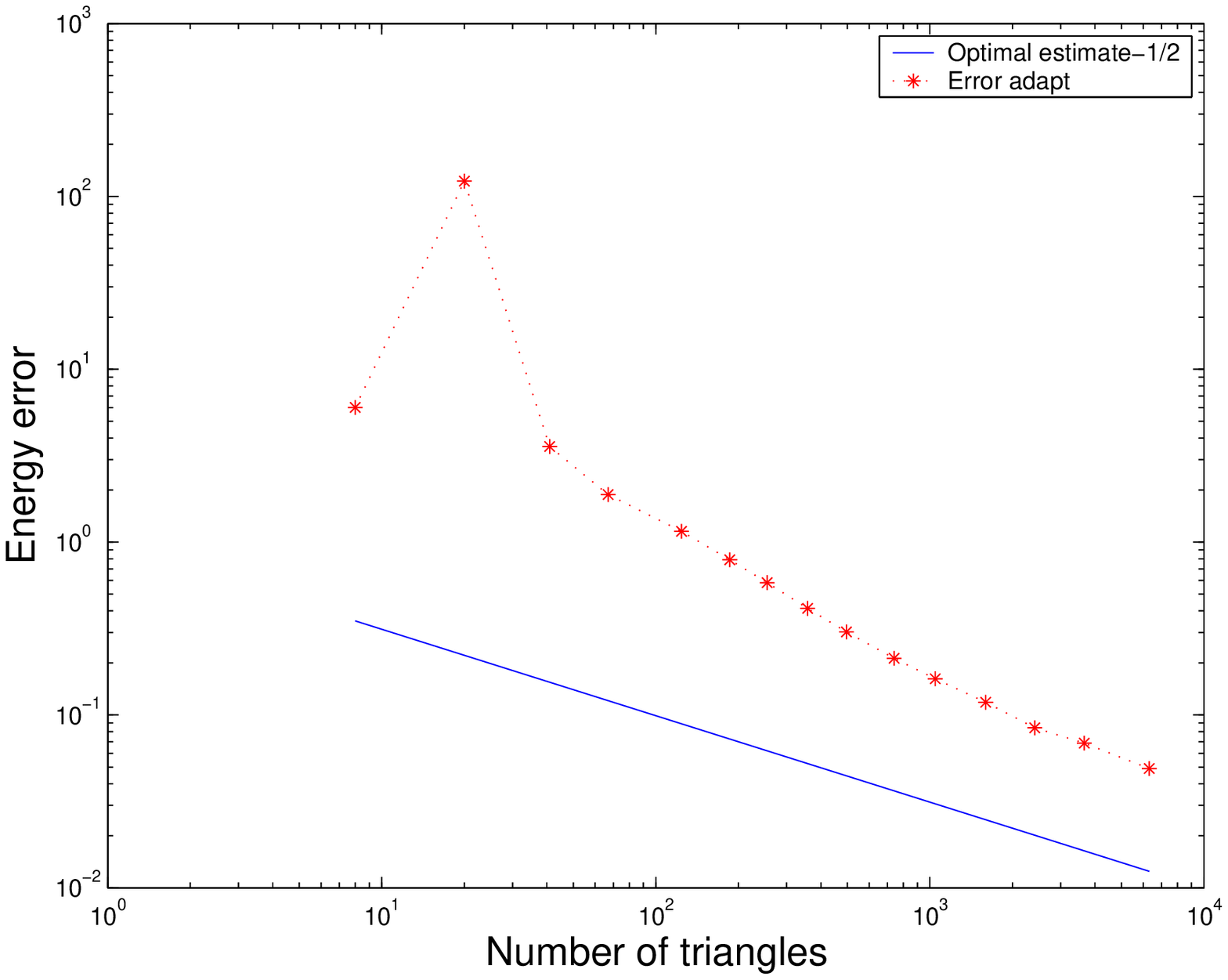}\\
  \end{minipage}
\addtocontents{lof}{figure}{FIG 6.5. {\small {\it  Actual error
against the number of elements in adaptively refined meshes for
$\varepsilon=0.1$ (left) and $\varepsilon=0.01$ (right) for the
marking parameter $\theta=0.5$.}}}\\
\end{figure}

Fig 6.4 shows the mesh with 10838 triangles (left) and the
postprocessing approximation to the scalar displacement $p$ on the
corresponding adaptively refined mesh (right) in the case
$\varepsilon=0.01$. Here the value of the postprocessing
approximation on each vertex is taken as the algorithmic mean of the
values of the displacement finite element solution on all the
elements sharing the vertex. The reason for the postprocessing is
that the displacement finite element solution is not continuous on
each vertex of the triangulation. We see that the refinement focuses
around boundary layers, which indicates that the estimators actually
capture boundary layers and resolve them in convection-domianed
regions. In addition, the postprocessing approximation to the scalar
displacement obtains  satisfactory results.

Fig 6.5 shows   the actual error (energy error) results
against the number of elements in adaptively refined meshes for
$\varepsilon=0.1$ (left) and $\varepsilon=0.01$ (right),
including two  theoretically-optimal order
(-1/2)   convergence lines. We see that in each case the actual error descends almost at the
optimal rate of convergence after several steps of iterations. The   numerical results confirm our theoretical
analysis.

\subsection{Convection-dominated model problem: interior and boundary layer}\ 
Set the domain $\Omega=[-1,1]\times[-1,1]$ with non-homogeneous
Dirichlet boundary conditions, the velocity field ${\bf w}=(2,1)$,
and the reaction term $r=0$ in (\ref{equation1}). The source term
$f=0$, the Dirichlet boundary conditions are as follows: $p=0 $ along the left and top
sides of the square and $p=100$ along the right and bottom sides.
The exact solution of this problem is unknown, but it is known that
it exhibits an exponential boundary layer at the boundary $x=1,y>0$
and a parabolic interior layer along the line connecting the points
$(-1,-1)$ and $(1,0)$.
\begin{figure}[htbp]
  \begin{minipage}[t]{0.5\linewidth}
    \centering
    \includegraphics[width=2.3in]{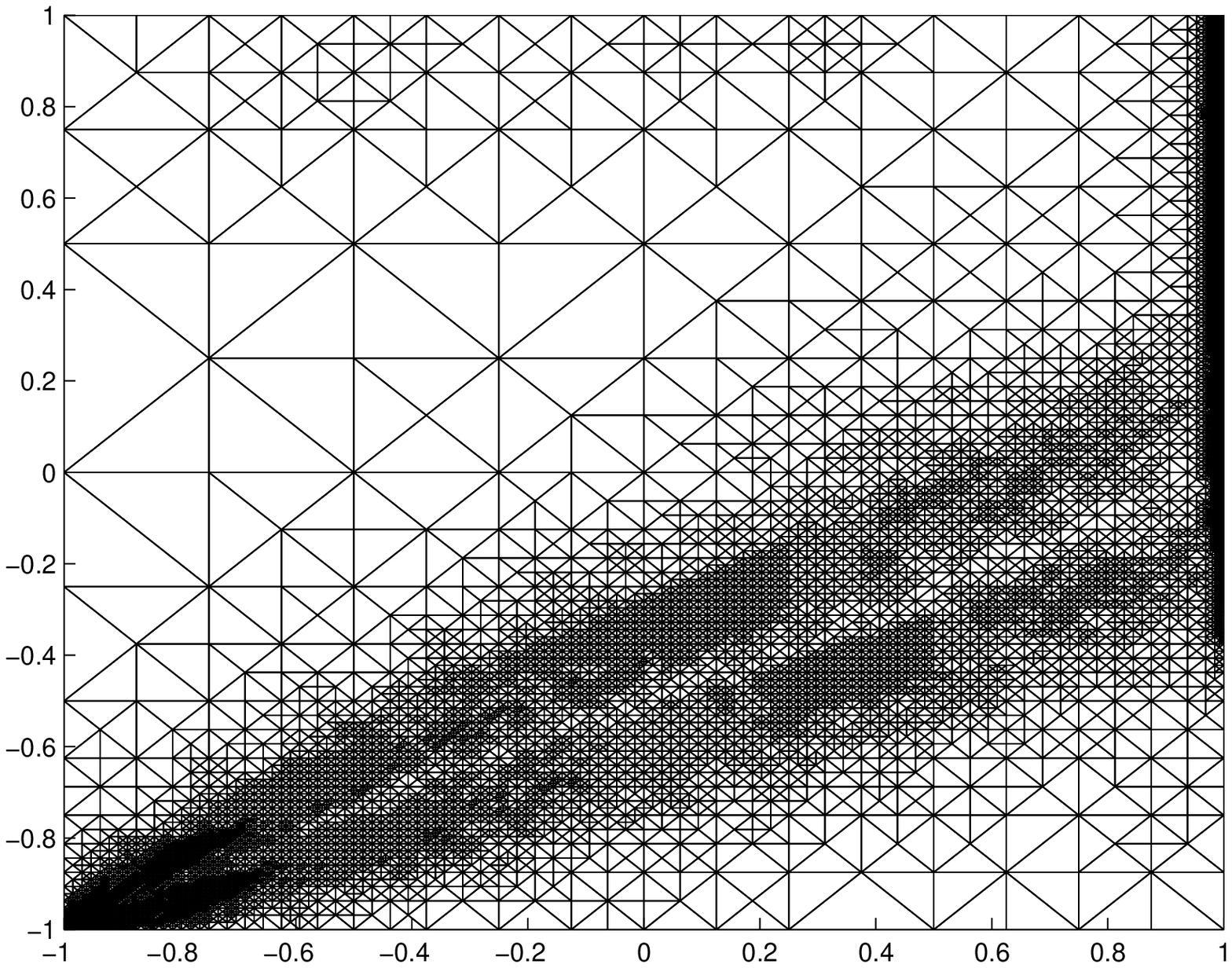}\\
  \end{minipage}
  \begin{minipage}[t]{0.5\linewidth}
    \centering
    \includegraphics[width=2.3in]{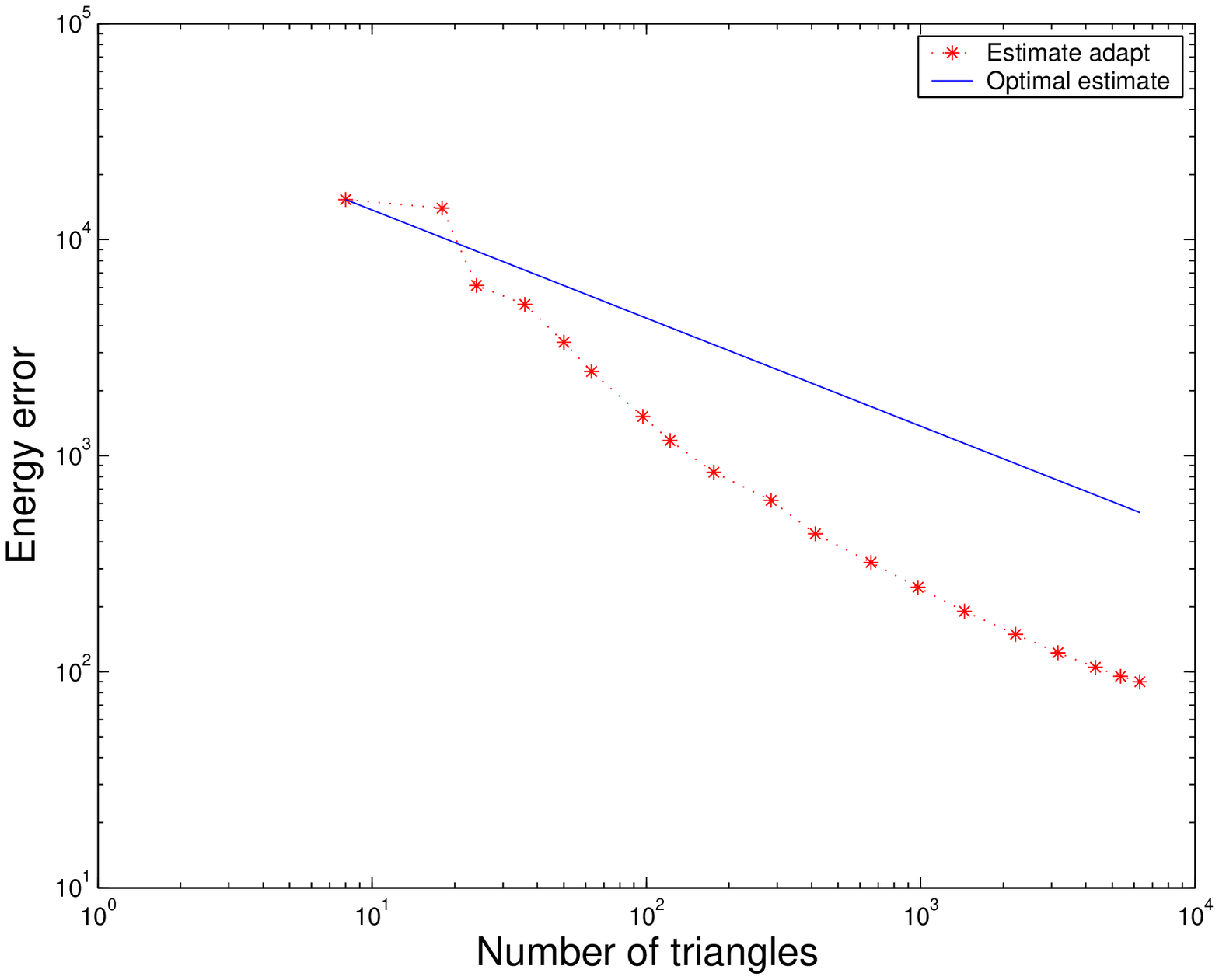}\\
  \end{minipage}
\addtocontents{lof}{figure}{FIG 6.6. {\small {\it A mesh with 47324
elements (left) for $\varepsilon=0.01,\theta=0.8$ and estimated
error against the number of elements in adaptively refined meshes
(right) for $\varepsilon=0.1,\theta=0.5$.}}}
\end{figure}

We still perform the AMFEM algorithm described in section 2 from the
origin mesh consisted of 8 right-angled triangles. From the left
graph of Fig 6.6, we can see that when using adaptive refinement the
mesh concentrates close to the exponential and parabolic layers. We note that
the refinement first occurs close to the region $x=1,y>0$, since
  the exponential layer is more stronger than the parabolic
layer.  The left
graph also illustrates that the a posteriori error estimator exactly
capture the behavior of the solution. The right
graph of Fig 6.6 shows that the estimated error rapidly reduces
starting from the fourth step of iterations, and  reaches the optimal
rate (-1/2) of convergence until the seventeenth step. This   convergence result is consistent with our theoretical analysis.

\end{document}